\documentclass[11pt]{article}
\usepackage[dvips]{graphicx}
\usepackage{latexsym}
\usepackage{amsmath,amssymb,amsfonts,amsthm}
\usepackage{xcolor}

\newcommand{\fkt}{{\cal F}_{kT}}

\newcommand{\gn}{{\cal G}_{n}}

\newcommand{\sig}{\sigma\sigma^*}

\newcommand{\nrn}{\rightarrow+\infty}
\newcommand{\xrn}{\xrightarrow}

\newcommand{\ER}{\mathbb {R}}
\newcommand{\EN}{\mathbb {N}}

\newcommand{\PE}{\mathbb {P}}
\newcommand{\ES}{\mathbb{E}}
\newcommand{\abs[1]}{\lvert#1\rvert}

\newcommand{\psg}{\langle }
\newcommand{\psd}{\rangle }
\newcommand{\cdeuxf}{\mathbf{(C_F^2)}}

\newcommand{\cunf}{\mathbf{(C_F^1)}}
\newcommand{\X}{{\bar{X}}}
\newcommand{\pnt}{{\cal P}^{(n,T)}} 
\newcommand{\usim}{\underset{\textnormal{\scalebox{1}{$^\sim$}}}}

\newcommand{\sapi}{\mathbf{(S_{a,p})}(i)}

\newcommand{\sap}{\mathbf{(S_{a,p})}}

\newcommand{\un}{\underline}
\newtheorem{theorem}{\textnormal{\bf{T\scriptsize{HEOREM}}}}
\newtheorem{prop}{\textnormal{\bf{P\scriptsize{ROPOSITION}}}}

\newtheorem{lemme}{\textnormal{\bf{L\scriptsize{EMMA}}}}

\theoremstyle{definition}

\theoremstyle{remark}

\newtheorem{Remarque}{\textnormal{\bf{R\scriptsize{EMARK}}}}

\author{Gilles Pag\`es\footnote{ Laboratoire de Probabilit\'es et Mod\`eles Al\'eatoires, UMR 7599, UPMC, Case 188, 4 pl. Jussieu,
F-75252 Paris Cedex 5, France, E-mail: \texttt{gilles.pages@upmc.fr}},
Fabien Panloup\footnote{Institut de Math\'ematiques de Toulouse, Universit\'e Paul Sabatier$\&$INSA Toulouse, 135, av. de Rangueil, F-31077 Toulouse Cedex 4, France, E-mail: \texttt{fabien.panloup@math.univ-toulouse.fr}}
}

\title{\textbf{Ergodic approximation of the distribution of a stationary diffusion : rate of convergence }}

\begin{document}
\maketitle
\begin{abstract}  We extend  to Lipschitz continuous functionals either of the true paths or of  the Euler scheme  with decreasing step of  a wide class of  Brownian ergodic  diffusions, the Central Limit Theorems  formally established for  their marginal empirical measure of these processes (which is classical for the diffusions and more recent as concerns their discretization schemes). We illustrate our results by simulations in connection with barrier option pricing.
\end{abstract}


\medskip
\noindent \textit{AMS classification (2000)}: 60G10, 60J60, 65C05, 65D15, 60F05.

\section{Introduction}
In a recent paper (\cite{PP1}), we investigated  weighted empirical measures based on some Euler schemes with
decreasing step in order to approximate recursively the distribution $\PE_\nu$ 
 of a stationary Feller Markov process $X:=(X_t)_{t\ge 0}$ with invariant distribution $\nu$.
To be precise, let
$(\bar{X}_t)$ be such an Euler scheme, let
$(\Gamma_k)_{k\ge 1}$ denote its sequence of discretization times and let $(\eta_k)_{k\ge 1}$ be a sequence of weights.
On the one hand we showed under some Lyapunov-type mean-reverting assumptions  on the
coefficients  of the $SDE$ and some  conditions on the steps and on the  weights that 
\begin{equation}\label{convarticle1}
{{\bar{\cal \nu}}}^{(n)}(\omega,F)=\frac{1}{\eta_1+\ldots+\eta_n}\sum_{k=1}^n \eta_k
F(\bar{X}_{\Gamma_k+.})\xrightarrow{n\rightarrow+\infty}\PE_\nu(F)=\int\ES[F(X^x)]\nu(dx)\qquad a.s.,
\end{equation}
for a broad class of functionals $F$  including bounded continuous
functionals for the Skorokhod topology. On the other hand, in the marginal case, $i.e.$ when $F(\alpha)=f(\alpha(0))$, then the procedure converges to $\nu(f)$. When the Poisson equation related to the infinitesimal generator has a solution, this convergence is ruled by a Central Limit Theorem ($CLT$): this has been extensively investigated in the literature (for continuous Markov processes, see~\cite{bhatta82}, for the Euler scheme with decreasing step of Brownian diffusions, see~\cite{LP1,lemaire1}). As concerns L\'evy driven $SDE$s, see~\cite{panloup3}.

\noindent Our aim in this paper is to extend some of  these rate results to functionals of the path process and its associated Euler scheme with decreasing step, $i.e.$ to study the rate of convergence to $\PE_\nu(F)$ of $(\frac{1}{t}\int_0^t F({X_{s+.}})ds)_{t\ge1}$ and  $({\bar{\cal \nu}}^{(n)}(\omega,F))_{n\ge1}$
respectively. Here, we choose to assume that $(X_t)$  is  an $\ER^d$-valued process solution to 
\begin{equation}\label{sde}
dX_t=b(X_t)dt+\sigma(X_t) dW_t,
\end{equation}
where $(W_t)_{t\ge0}$ is  a ${q}$-dimensional Brownian motion and $b$ and $\sigma$ are {\em Lipschitz continuous functions} with values in $\ER^d$ and $\mathbb{M}_{d,{q}}$ respectively, where $\mathbb{M}_{d,{q}}$ denotes the set of $d\times {q}$-matrices. Under these assumptions, strong existence and uniqueness hold and $(X_t)$ is a Markov process whose semi-group is denoted by $(P_t)$. We also assume that $(X_t)$ has a unique invariant distribution $\nu$ and we denote by $\PE_\nu$, the distribution of $(X_t)$ when stationary.

\medskip
Let us now focus on the discretization of $(X_t)$. We are going to introduce some continuous-time Euler schemes with decreasing step:  denoting by $(\Gamma_n)_{n\ge1}$  the {\em increasing} sequence of discretization times starting from $\Gamma_0=0$, we assume that the  step sequence defined by 
$\gamma_n:=\Gamma_n-\Gamma_{n-1}$, $n\ge 1$,  is  {\em nonincreasing } and satisfies 
 \begin{equation}\label{nonincreagam}
\lim_{n\rightarrow+\infty}\gamma_n=0,\quad
\Gamma_n= \sum_{k=1}^n\gamma_k\xrightarrow{n\nrn}+\infty,
\end{equation}
First, we introduce   the discrete time  constant Euler scheme $(\bar{X}_{\Gamma_n})_{n\ge0}$ recursively defined at the discretization  times $\Gamma_n$ by $\bar X_0= x_0$ and 

\begin{equation}\label{Eulerscheme}
{\bar{X}}_{{\Gamma_{n+1}}}=\bar{X}_{{\Gamma_{n}}}+\gamma_{n+1}
b(\bar{X}_{{\Gamma_{n}}})+\sigma(\bar{X}_{{\Gamma_{n+1}}})(W_{\Gamma_{n+1}}-W_{\Gamma_{n}}).
\end{equation}

\noindent There are several ways to extend this definition into a continuous time process. The simplest one  is the stepwise constant Euler scheme $(\bar{X}_t)_{t\ge0}$ defined by
\[
\forall\, n\!\in \EN,\quad \forall\,t \in [\Gamma_n,\Gamma_{n+1}),\quad \bar X_t =\bar X_{\Gamma_n}.
\]
The stepwise constant Euler scheme is a right continuous left limited process (referred as c\`adl\`ag throughout the paper, following the French acronym). This scheme is easy to simulate provided one is able to compute the functions $b$ and $\sigma$ at a reasonable cost.  One could also introduce the linearly interpolated process built on $(\bar{X}_{\Gamma_n})_{n\ge0}$ but except the fact that it is a continuous process, it has no specific virtue in term of simulability or convergence rate.

The second possibility to extend the discrete time Euler scheme is what we will call  the  \textit{genuine} Euler scheme, denoted  from now on by   $(\xi_t)_{t\ge0}$. It is defined by interpolating the two part of the discrete time scheme in its own scale (time, Brownian motion). It is defined by
\begin{equation}\label{genuine}
\forall\, n\!\in \EN,\quad \forall\,t \in [\Gamma_n,\Gamma_{n+1}),\quad \xi_t=\bar X_{\Gamma_n}+(t-\Gamma_n)b(\bar X_{\Gamma_n}) + \sigma(\bar X_{\Gamma_n})  (W_{t}-W_{\Gamma_{n+1}})   
\end{equation}

Such an approximation looks more accurate than the former one, especially in a functional setting, as it has been emphasized --~in a constant step framework~-- in the literature on several problems related to the Monte Carlo estimation of ($a.s.$ continuous) functionals of a diffusion (with a finite horizon) (see~$e.g.$~\cite{BOLE}, Chapter~5). This follows from the classical fact that the $L^p$-convergence rate of this scheme for the sup norm is $\sqrt{\gamma}$ instead of $-\sqrt{\gamma}\log \gamma$ for its stepwise constant counterpart (where $\gamma$ stand for the step). On the other hand, the simulation of a functional of $(\xi_t)_{t\in [\tau,\tau +T]}$ is deeply connected with the  simulation of the Brownian bridge so that it is only possible for specific functionals (like running maxima, etc).

A convenient and synthetic form for the genuine Euler scheme is to write it as an It\^o process satisfying the following pseudo-diffusion equation
\begin{equation}\label{genuine2}
\xi_t=x_0+\int_0^t b(\xi_{\un{s}})ds+\int_0^t \sigma(\xi_{\un{s}})dW_s 
\end{equation}
where
\begin{equation}\label{sbar} 
\un{t}=\Gamma_{N(t)}\quad\textnormal{with}\quad N(t)=\min\{n\ge 0, \Gamma_{n+1}>t\}.
\end{equation}

\noindent Taking advantage of this notation for the stepwise constant Euler scheme, one can also note that 
\[
\forall\, t\in \ER_+,\qquad \bar X_t = \bar X_{\un{t}}.
\]
When  necessary, we will adopt the more precise notation $\X^{x,(h_n)}$ for a stepwise constant continuous-time Euler scheme to specify starting at $x\in\ER^d$ at time $0$ with  a nonincreasing step sequence $(h_n)_{n\ge1}$ satisfying~(\ref{nonincreagam}).

%


\medskip
Since we will deal with possibly c\`adl\`ag approximations of continuous processes we will introduce the spaces $\mathbb{D}_{uc}(I,\ER^d)$  of $\ER^d$-valued c\`adl\`ag functions on $I=\ER_+$ or $[0,T]$, $T>0$, endowed with the topology of the uniform convergence on compact sets, rather than the classical Skorokhod topology (see~\cite{billingsley}). In fact, one must keep in mind that if $\alpha:I\to \ER^d$ is a continuous function and $(\alpha_n)$ is a sequence of c\`adl\`ag functions, $\alpha_n \stackrel{Sk}{\to}\alpha$ iff $\alpha\stackrel{uc}{\to}\alpha$ (with obvious notations). Furthermore, usual Skorokhod distance $d_{Sk}$ (so-called $J_1$ and $J_2$ topologies)   on $\mathbb{D}([0,T],\ER^d)$   all satisfy 
$$
d_{Sk}(\alpha,\beta)\le \|\alpha-\beta\|_{_T}:=\sup_{t\in [0,T]}|\alpha(t)-\beta(t)| 
$$
so that  any functional $F:\mathbb{D}([0,T],\ER^d)\to \ER$ which is Lipschitz with respect to such a distance $d_{Sk}$ will be Lipschitz continuous with respect to $\|\,.\,\|_{_T}$ (hence measurable with respect  to the Borel $\sigma$-field induced by the Skorokhod topology).

\noindent At this stage, we need to introduce further notations related to the long run behaviour of processes (or simply functions). Let $\delta_{\alpha}(d\beta)$ denotes the Dirac mass at $\alpha\!\in \mathbb{D}(\ER_+,\ER^d)$ and $\alpha^{(u)}:=(\alpha_{u+t})_{t\ge0}$ denotes  the $u$-shift of $\alpha$. 

\smallskip
\noindent We will see below that our aim is to elucidate the asymptotic $\PE(d\omega)$-$a.s.$ weak behaviour of  the empirical measures $\displaystyle  \frac 1t \int_0^t \delta_{Y^{(s)}(\omega)}(d\beta)ds $ as $t$ goes to infinity, where $Y$ will be the diffusion $X$ itself or one of its (simulatable) Euler time discretizations.   This suggests to introduce a time dicretization at times $\Gamma_n$  of the above time integral like we did to define the Euler scheme. This leads us to introduce, for any $\alpha\!\in \mathbb{D}(\ER_+,\ER^d)$, the following abstract ``Euler'' empirical means
\[
\bar{\nu}^{(n)}(\alpha,d\beta) = \frac{1}{\Gamma_n} \sum_{k=1}^n \gamma_k\,\delta_{\alpha^{(\Gamma_{k-1})}}(d\beta) =  \frac{1}{\Gamma_n}\int_0^{\Gamma_n} \delta_{\alpha^{(\un{s})}}(d\beta)ds.
\]
Then, for a functional $F$ defined on  $\mathbb{D}(\ER_+,\ER^d)$ and $\alpha\in\mathbb{D}(\ER_+,\ER^d)$, 
 \[
 \bar{\nu}^{(n)}(\alpha,F)= \int_{\mathbb{D}(\ER_+,\ER^d)}F(\beta)  \bar{\nu}^{(n)}(\alpha,d\beta)= \frac{1}{\Gamma_n}\sum_{k=1}^n \gamma_k F(\alpha^{(\Gamma_{k-1})})= \frac{1}{\Gamma_n} \int_0^{\Gamma_n} F(\alpha^{(\un{s})})ds.
 \]
 
\noindent In the following, we will use this sequence of empirical measures for both stepwise constant and genuine Euler schemes. Compared to~\cite{PP1}, this means that we assume that the sequence of weights $(\eta_n)$ satisfies $\eta_n=\gamma_n$ for every $n\ge1$. 

\bigskip
\noindent{\sc Additional notations.} $\rhd$ $\psg x,y\psd=\sum_ix_iy_i $ will denote the canonical  inner product  and $|x|= \sqrt{\psg x,x\psd}$ will denote Euclidean norm of a vector $x\!\in \ER^d$.

\noindent $\rhd$ Let $A=[a_{ij}] \!\in \mathbb{M}_{d,q}$ be an $\ER$-valued matrix with $d$ rows and $q$ columns. $A^*$ will denote the transpose of $A$, ${\rm Tr}(A)=\sum_i a_{ii}$ its trace and $\|A\|:= \sqrt{{\rm Tr}(AA^*)}= (\sum_{ij}a^2_{ij})^{\frac12}$. If $d=q$, one writes $A x^{\otimes 2}$ for $x^*Ax$.

\section{Main results}
\subsection{Assumptions and background}
We denote by $({\cal F}_t)_{t\ge0}$ the usual augmentation of $\sigma(W_s,0\!\le\! s\!\le\! t)$ by $\PE$-negligible sets. Since $b$ and $\sigma$ are Lipschitz continuous functions, Equation~(\ref{sde}) admits  a 
 unique $({\cal F}_t)$-adapted solution $(X_t^x)_{t\ge0}$ starting from $x\in\ER^d$.
More generally, for every $u\ge 0$ and every finite ${\cal F}_u$-measurable random variable $\Xi$, we can consider  $(X^{(u),\Xi}_{t})_{t\ge0}$, unique strong solution to the $SDE$:
\begin{equation} \label{equationshiftee}
 dY_t=b(Y_t)dt+\sigma(Y_t)dW_t^{(u)},\quad Y_0=\Xi,
\end{equation}
where $W_t^{(u)}=W_{u+t}-W_u$, $t\!\ge\!0$,  is the $u$-shifted Brownian motion (independent of ${\cal F}_u$). Note that $X^x_t=X^{(0),x}_{t}$ and that $X^{(u),\Xi}_t$ can be also defined through the flow of~(\ref{sde}) by setting
\[
X^{(u),\Xi}_{t}=\big(X^{(u),x}_t\big)_{|x=\Xi}.
\]
\noindent Throughout this paper, we consider a measurable functional $F:\mathbb{D}_{uc}([0,T],\ER^d)\rightarrow\ER$. We will denote by $F_{_T}$ the stopped functional defined on $\mathbb{D}_{uc}(\ER_+,\ER^d)$   by
\begin{equation}\label{functionalt}
\forall\, \alpha\in\mathbb{D}_{uc}(\ER_+,\ER^d),\qquad F_{_T}(\alpha)=F(\alpha^T)\quad\textnormal{with}\quad \alpha^T(t)=\alpha(t\wedge T),\quad t\ge0.
\end{equation}  
Let us introduce the assumptions on $F$.

\medskip
\noindent $\cunf$: $F:\mathbb{D}_{uc}([0,T],\ER^d)\rightarrow\ER$ is a bounded and Lipschitz continuous functional.

\medskip
\noindent We set 
$$
f_{_F}(x)=\ES[F_{_T}(X^x)]=\ES[F(X^x_t,0\le t\le T)].
$$
It is classical background (see $e.g.$~\cite{kunita}) that, under the Lipschitz assumption on $b$ and $\sigma$, 
$\ES[\sup_{t\in [0,T]}|X^x_t-X^y_t|]\le C_{b,\sigma,T} |x-y|$ so that $f_{_F}$ is in turn clearly Lipschitz continuous.
Additional regularity properties (like differentiability) can be transfered from $f_F$ provided $F$, $b$ and $\sigma$ are themselves differentiable enough (see $e.g.$~\cite{kunita}). Furthermore, it follows from its very definition and the Markov property that 
\[
\nu(f_{_F}) = \PE_{\nu}(F_{_T})=\int\ES[F_{_T}(X^x)]\nu(dx).
\]
\noindent $\cdeuxf$: There exists   a bounded  ${\cal C}^2$-function $g_{_F}:\ER^d\rightarrow\ER$ with bounded Lipschitz continuous derivatives such that 
$$
 \forall x\in\ER^d,\qquad f_{_F}(x)-\nu(f_{_F})={\cal A}g_{_F}
$$
where ${\cal A}$ denotes the infinitesimal generator of the diffusion~(\ref{sde})
defined for every $ {\cal C}^2$-function $f$ on $\ER^d$ by 
$$ {\cal A}f(x)=\psg\nabla f,b\psd(x)+\frac{1}{2}{\rm Tr}(\sigma^*D^2f\sigma(x)).$$

\begin{Remarque}  In fact, we need in the sequel that $f_F$ satisfies a $CLT$ for the marginal occupation measures which follows (see~\cite{LP1, panloup3}) from Assumption $\cdeuxf$ combined with a Lyapunov stability assumption (such as $\mathbf{(S_{a,p})}$ introduced below).
Namely, we have for a class of regular  functions $f$ satisfying $f={\cal A}g+C$
\begin{equation}\label{TCLmarg1}
\sqrt{t}\left( \frac{1}{t}\int_0^t f(X_s^x)ds-\nu(f)\right)\xrn[t\rightarrow+\infty]{{\cal L}}{\cal N}\left(0,\sigma^2_f\right)
\end{equation}
and as soon as  $\displaystyle \sum_{k=1}^n \frac{\gamma_k^2}{\sqrt{\Gamma_k}}\xrn{n\nrn}0,$
\begin{equation}
\sqrt{\Gamma_n}\left(\frac{1}{\Gamma_n}\sum_{k=1}^n \gamma_k f(\bar{X}_{\Gamma_{k-1}})-\nu(f)\right)\xrn[n\rightarrow+\infty]{\cal L}{\cal N}\left(0,\sigma^2_f\right),\label{TCLmarg2}
\end{equation}
 where 
 $$
 \sigma^2_f=\int_{\ER^d} |  \sigma^*\nabla g (x) |^2\nu(dx)=-2\int g(x) {\cal A}g(x)\nu(dx)
 $$ 
 and ${\cal L}$ denotes the weak convergence of (real valued) random variables.  For details on results in these directions, see~\cite{bhatta82} for the continuous case and ~\cite{LP2,lemaire1,panloup3} for the decreasing step Euler scheme.
 
 \smallskip
 Checking when Assumption $\cdeuxf$ is fulfilled is equivalent to solve the Poisson equation ${\cal A}u=f$ on $\ER^d$. When $f$ has compact support, well-known results about the same equation in a bounded domain lead to Assumption $\cdeuxf$ when the diffusion is uniformly elliptic
 (see $e.g.$ \cite{lady}, Theorems III.1.1 and III.1.2). Such an assumption on $f_{_F}$ is clearly unrealistic. In the general case, in \cite{Parver1}, \cite{Parver2} and \cite{Parver3}, the problem is solved  under some ellipticity conditions in some Sobolev spaces and controls of the growth are given for $u$ and its first derivatives. Finally, when the diffusion is an Ornstein-Uhlenbeck process, one can refer to \cite{LP1} where the problem is solved
  in ${\cal C}^2(\ER^d)$. 
 %
\end{Remarque}

\smallskip
\noindent Let us now introduce the Lyapunov-type stability assumptions on  $SDE$~\eqref{sde}. Let ${\cal E}\!{\cal Q}(\ER^d)$ denote the set of {\em Essentially Quadratic} functions, that is ${\cal C}^2$-functions
$V:\ER^d\rightarrow (0,\infty)$ such that 
\[
\lim_{|x|\rightarrow+\infty} V(x)=+\infty, \qquad \abs[\nabla V]\le C \sqrt{V}\quad \mbox{  and }\quad D^2V \mbox{ is  bounded.}
 \]
\noindent Note that since $V$ is continuous, $V$ attains its positive minimum $\underline{v}>0$ so that,  for any $A, r>0$, there exists a real constant $C_{_{A,r}}$ such that $A+V^r \le C_{_{A,r}}V^r$.

\noindent Let us come to the  mean-reverting assumption itself. First, for any symmetric $d\times d$ matrix $S$, set  $\lambda^+_S:= \max(0,\lambda_1,\ldots,\lambda_d)$ where $\lambda_1,\ldots,\lambda_d$ denote  the eigenvalues of $S$. Let 
$a\in(0,1]$ and $p\!\in [1,\,+\infty)$. We introduce the following mean-reverting assumption {\em  with intensity $a$}: 
\begin{flushleft}
$\mathbf{(S_{a,p}):}$ There exists a function $V\in{\cal E}\!{\cal
Q}(\ER^d)$ such that:
\begin{align*}
(i)& \quad   \textnormal{$\exists C_a>0$ such that } |b|^2+{\rm Tr}(\sig)\le C_aV^a.
&\\
(ii)&\quad \textnormal{There exist $\beta\in\ER$ and $\rho>0$
such that }\psg\nabla V,b\psd+\lambda_p{\rm Tr}(\sig)\le \beta-\rho V^a,
\end{align*}
\end{flushleft}
where $\displaystyle{\lambda_p:=\frac{1}{2}\sup_{x\in\ER^d}\lambda^+_{D^2V(x)+(p-1)\frac{\nabla V\otimes\nabla V}{V}}}$. The function $V$ is then called  a Lyapunov function for the diffusion $(X_t)_{t\ge 0}$.

\smallskip
In Theorem 3 of~\cite{LP2}, it is shown that this assumption leads to an $a.s.$ marginal weak convergence result to the set of invariant distributions of the diffusion. When $p\ge2$ and the invariant distribution is unique, this result reads as follows.
\begin{prop}
\label{funcpremiere1} Let  $a\in(0,1]$ and $p\ge2$ such that
$\mathbf{(S_{a,p})}$ holds. Then, 
\begin{equation}\label{eq:rappellyap}
\sup_{n\ge1}\frac{1}{\Gamma_n}\sum_{k=1}^n\gamma_k V^{\frac{p}{2}+a-1}(\bar{X}_{\Gamma_{k-1}})<+\infty\quad a.s.
\end{equation}
Let $\nu$ denote the unique invariant distribution of~\eqref{sde}. Then,  $a.s.$,
$$
\frac{1}{\Gamma_n}\sum_{k=1}^n\gamma_k f(\bar{X}_{\Gamma_{k-1}})\xrightarrow{n\rightarrow+\infty}\nu(f)
$$
for every continuous function $f$ satisfying $f(x)=o(V^{\frac{p}{2}+a-1}(x))$ as $|x|\rightarrow+\infty$.
\end{prop}

\begin{Remarque} In the case $V(x)=1+|x|^2$, one checks for instance that for a given $a\in(0,1]$, Assumption $\mathbf{(S_{a,p})}$ is fulfilled for every $p\ge1$ if ${\rm Tr}(\sigma\sigma^*)(x)=o(1+|x|^{2a})$ as $|x|\rightarrow+\infty$ and
$$
b(x)=-\rho(x)\frac{x}{|x|}+\mathcal{T}(x)\qquad\textnormal{where $C_1 |x|^{2a-1}\le\rho(x)\le C_2|x|^{2a-1},$ }
$$
and $\mathcal{T}$ satisfies for every $ x\in\ER^d$ $\psg \mathcal{T}(x),x\psd=0$  and $|\mathcal{T}(x)|\le C(1+|x|^a)$.
\end{Remarque}
\medskip 
\noindent 
As concerns the uniqueness of the invariant distribution $\nu$, we need an additional assumption related to the transition $P_{_T}$. Namely, we assume that:

\smallskip
\noindent $\mathbf{(S^\nu_T)}$: $\nu$ is an invariant distribution for $(P_t)_{t\ge0}$ and the unique one for $P_{_T}$.   

\smallskip
\noindent Then, $\nu$ is in particular the unique invariant distribution for $(P_t)_{t\ge0}$. In fact, checking uniqueness of the invariant distribution for $P_{_T}$ at a given time $T>0$ is a standard way to establish uniqueness for the whole semi-group $(P_t)_{t\ge0}$. To this end, one may use the following two typical criterions:

\smallskip
$\bullet$ Irreducibility based on ellipticity: for every $x\in\ER^d$, $P_{_T}(x,dy)$ has a density $(p_{_T}(x,y))_{y\in\ER^d}$ $w.r.t.$ the Lebesgue measure $\lambda_d$ and
$\lambda_d(dy)-a.s.$, $p_{_T}(x,y)>0$ for every $x\in\ER^d$. 

\smallskip
  $\bullet$ Asymptotic confluence:  for every bounded Lipschitz continuous function $f$, for every compact subset $K$ of $\ER^d$, 
$$
\sup_{(x_1,x_2)\in K} \left|P_{kT}f(x_1)-P_{kT}f(x_2)\right|\xrn{k\nrn}0 \quad\textnormal{(see $e.g.$~\cite{basak, lemaire1}).}
$$




\subsection{Main results}
\noindent We are now in position to state our main results. 
\begin{theorem} \label{thprincipal} Let $T>0$. Assume $b$ and $\sigma$ are Lipschitz continuous functions
 satisfying  $\mathbf{(S_{a,p})}$ with an essentially quadratic Lyapunov function $V:\ER^d\to (0,+\infty)$ and parameters $a\in(0,1]$ and $p>2$.  Assume furthermore that $V$ satisfies the   growth assumption:
\begin{equation}\label{gaV}
\liminf_{|x|\rightarrow+\infty} \frac{V^{p+a-1}(x)}{|x|}>0.
\end{equation}
Assume that the uniqueness assumption $\mathbf{(S^\nu_T)}$ holds. Finally, assume that the step sequence $(\gamma_n)_{n\ge 1}$ satisfies~(\ref{nonincreagam}) and 
\begin{equation}\label{condpas33}
\sum_{k\ge1}\frac{\gamma_k^{3/2}}{\sqrt{\Gamma_k}}<+\infty.
\end{equation}
Let $F:\mathbb{D}_{uc}([0,T],\ER^d)\to \ER$ be a functional satisfying $\mathbf{(C_F^1)}$ and $\mathbf{(C_F^2)}$. 

\smallskip
\noindent $(a)$ {\sc  Genuine Euler scheme}:  Then 
\begin{equation}\label{cvloi2}
\sqrt{\Gamma_n}\left(\bar{\nu}^{(n)}(\xi(\omega),F_{_T})-\PE_\nu(F_{_T})\right)\xrn[n\rightarrow+\infty]{\cal L}{\cal N}\left(0,\sigma^2_{_F}\right),
\end{equation}
where
\begin{equation}\label{eq:sigmacarre}
\sigma^2_{_F}=\frac{1}{T}\left(\int\ES\left[\Big(\ES\Big(A_{_{2T}}^x\,|\,{\cal F}_{_{2T}}\Big)-\ES\Big(A_{_T}^x\,|\,{\cal F}_{_T}\Big)-\int_T^{2T}\sigma^*\nabla g_F(X_u^x)dW_u\Big)^2\right]\nu(dx)\right)
\end{equation}
and $A^x_t:=\int_0^t \left(F_{_T}(X^x_{u+.})-f_{_F}(X^x_u)\right) du$, $t\ge 0$, (is ${\cal F}_{t+T}$-adapted).

\smallskip
\noindent $(b)$ {\sc  Stepwise constant Euler scheme}: furthermore, if there exists $\delta>0$ such that
\begin{equation}\label{condpas34}
\sum_{k\ge1}\frac{\gamma_k^{\frac{3}{2}-\delta}}{\sqrt{\Gamma_k}}<+\infty,
\end{equation}
then, 
\begin{equation}\label{cvloi}
\sqrt{\Gamma_n}\left(\bar{\nu}^{(n)}(\bar{X}(\omega),F_{_T})-\PE_\nu(F_{_T})\right)\xrn[n\rightarrow+\infty]{\cal L}{\cal N}\left(0,\sigma^2_{_F}\right).
\end{equation}

\end{theorem}

\begin{Remarque} \label{remarquecalcul}By a series of computations, we can obtain other expressions for $\sigma^2_{_F}$. In particular, 
we check in the Appendix~A  that $\sigma^2_{_F}$ reads
\begin{equation}\label{eq:sigmacarre2}
\sigma^2_{_F}=2\int_{0}^{T}(1-\frac vT)C_F(v)dv-2 \ES_{\nu}\Big(F_{_T}(X)\int_0^T\sigma^*\nabla g_{_F}(X_u)dW_u\Big)+  \int_{\ER^d} |  \sigma^*\nabla g_{_F} (x) |^2\nu(dx),
\end{equation}
where $\ES_{\nu}$ denotes the expectation under the stationary regime and $C_{_F}$ is the covariance function defined by 
\begin{equation}\label{covariancefunction}
C_{_F}(u)= \ES_{\nu}\big(F_{_T}(X_{u+.})-f_{_F}(X_u))(F_{_T}(X)-f_{_F}(X_0))\big).
\end{equation}
This expression is not clearly positive but has the advantage to separate the ``marginal part'' that is represented by the last term from the ``functional part'' which corresponds to the first two ones.

\smallskip 
\noindent For instance, when  $F(\alpha)=\phi(\alpha(0))$, $\phi$ being bounded and such that
$\phi-\nu(\phi)= {\cal A}h$ where $h$ is a bounded ${\cal C}^2$-function with bounded derivatives, then $f_{_F}=\phi$ and one observes that the first  two  terms of~\eqref{eq:sigmacarre2} are equal to $0$ so that $\sigma^2_{F} = \displaystyle \int_{\ER^d}|\sigma^*\nabla g_{_F}(x)|^2\nu(dx)$. This means that we retrieve the marginal $CLT$ given by~\eqref{TCLmarg2} (under a slightly more condition on the step sequence which is  adapted to the more general functionals we are dealing with,  thus,  more constraining than that of the original paper; see below for more detailed comments on the steps conditions).

\smallskip
If we now consider $F_T$ defined $F_{_T}(\alpha)=\phi(\alpha(T))$, $\phi$ satisfying the same assumptions as before, one can straightforwardly deduce from a simple change of variable that the limiting variance is still  
$\int_{\ER^d} |  \sigma^*\nabla h (x) |^2\nu(dx)$. In the appendix (Part B), we show that retrieving this limiting variance using~\eqref{eq:sigmacarre} is possible but requires some non trivial computations. In particular, this calculus emphasizes the intricate nature of the structure of the functional variance. 
\end{Remarque}
%
%
Given the form of $\bar{\nu}^{(n)}$, it seems natural to introduce the (non-\textit{simulatable}) sequence  
$$
\frac{1}{\Gamma_n}\int_{0}^{\Gamma_n}F_{_T}(\xi^{({u})})du
$$
which in fact appears naturally as a tool in the proof of the above theorem.
\begin{theorem} \label{thprincipal3} Assume  the assumptions of Theorem~\ref{thprincipal}(a). Then,

\begin{equation}\label{cvloi3}
{\sqrt{t}}\left(\frac{1}{t}\int_0^tF_{_T}(\xi^{(s)})ds-\PE_\nu(F_{_T})\right)\xrn[n\rightarrow+\infty]{\cal L}{\cal N}\left(0,\sigma^2_{_F}\right).
\end{equation}
\end{theorem}

Finally, we also state the Central Limit Theorem  for the stochastic process $(X_t)_{t\ge0}$ itself. This result can be viewed as a (partial)  extension to functionals of Bhattacharya's  $CLT$ established in~\cite{bhatta82} for a class of ergodic Markov processes.

\begin{theorem} \label{thprincipal2} Let $T>0$. Assume $b$ and $\sigma$  are Lipschitz continuous functions 
satisfying  $\mathbf{(S_{a,p})}$ with an essentially quadratic Lyapunov function  $V$ and parameters $a\in(0,1]$ and $p>2$. Assume $\mathbf{(S^\nu_T)}$ holds. Let $F:\mathbb{D}_{uc}([0,T],\ER^d)\to \ER$ be a functional satisfying $\mathbf{(C_F^1)}$ and $\mathbf{(C_F^2)}$. Then, for every $x\in\ER^d$,
\begin{equation}\label{cvloi3}
{\sqrt{t}}\left(\frac{1}{t}\int_0^tF(X_{u}^{(s),x},0\le u\le T)\,ds-\PE_\nu(F_{_T})\right)\xrn[n\rightarrow+\infty]{\cal L}{\cal N}\left(0,\sigma^2_{_F}\right).
\end{equation}
\end{theorem}

This means that our approach (averaging decreasing step schemes) induces no loss of weak rate of convergence with respect to that of the empirical mean of the process itself towards its steady regime. If we look at the problem from an algorithmic point of view, the situation becomes quite different. First, we will no longer discuss the recursive aspects as well as the possible storing  problems induced by the use of decreasing steps: it has already been done in~\cite{PP1} and we showed that they can easily be encompassed in practice, especially for additive  functionals 
or  functions of running extrema (see $e.g.$ simulations in Section~\ref{simulations}). 

\smallskip
Our aim here is to discuss the rate of convergence in terms of complexity. It is clear from its design that the complexity of the algorithm grows  linearly with the number of iterations. Thus, if $\gamma_n \propto n^{-\rho}$, $0<\rho< 1$,  then $\Gamma_n \sim \frac{n^{1-\rho}}{1-\rho}$  so that the effective rate of convergence as a function of the complexity is essentially proportional to $n^{\frac{1-\rho}{2}}$. However, the choice of $\rho$ is constrained by conditions~\eqref{condpas33} or~\eqref{condpas34} that are required for the control of the discretization error. These conditions imply that $\rho$ must be taken greater than 1/2 and lead to an ``optimal'' rate proportional to $n^{\frac{1}{4}-\varepsilon}$ for every $\varepsilon>0$. This means that we are not able to recover the  optimal rate  of the marginal case that is proportional to $n^{-1/3}$ and obtained for $\rho=1/3$ (see~\cite{LP2} for details). Indeed, in this functional framework, the weak discretization  error is generally smaller and thus, is negligible  compared to the long time error under a more constraining step condition~\eqref{condpas33} instead of $\sum\gamma_k^2/{\sqrt{\Gamma_n}}<+\infty$ in the marginal case).

\smallskip
The 
paper is organized as follows. In Sections~\ref{preliminaries},~\ref{sectionmartingale} and~\ref{Thetastudy}, we will focus on the proof of Theorem~\ref{thprincipal}$(a)$ and Theorem~\ref{thprincipal2} about the rate of convergence of the two considered occupation measures of the genuine Euler scheme. Then, in Section~\ref{PMT}, we will summarize the results of the previous sections and will give the main arguments of the proof of Theorems~\ref{thprincipal}$(a)$ and~\ref{thprincipal2}. Finally, Section~\ref{simulations} is devoted to  numerical tests in a financial framework: the pricing  of a barrier option when the underlying asset price dynamics is  a stationary stochastic volatility model. 

\section{Preliminaries}\label{preliminaries}
As for  the marginal rate of convergence (see~\cite{LP1}), the first idea is to find a good decomposition of the error (see Lemma~\ref{lemme1}). In particular, we have to exhibit a main martingale component. Here, since $F$ depends on the trajectory of the process between $0$ and $T$, the idea is that the ``good'' filtration for the main martingale component is 
$({\cal F}_{{kT}})_{k\ge0}$. That is why,  in the main part of the proof of these theorems, we will introduce and study the sequence of random
probabilities $({\cal P}^{(n,T)}(\omega,d\alpha))_{n\ge 1}$ defined  by:
$$
{\cal P}^{(n,T)}(\omega,d\beta)=\frac{1}{nT}\int_0^{nT}\delta_{\xi^{(\usim{u})}}(d\beta)du
=\frac{1}{nT}\sum_{k=1}^n\int_{(k-1)T}^{kT}\delta_{\xi^{(\usim{u})}}(d\beta)du
$$
where $\usim{u}$ is a deterministic real number lying in $[\un{u}, u]$.
 
\noindent To alleviate the notations, we will denote from now on,   ${\cal G}_k=\fkt$ and $\ES_k[\,.\,]=\ES[.\, |\,{\cal G}_k]$,  $k\ge0$.

At this stage, the reader can observe on the one hand  that for a bounded functional  $F$, ${\cal P}^{(n,T)}(\omega,F_{_T}) $ is  ${\cal G}_{n+1}={\cal F}_{(n+1)T}$-adapted for every $n\ge 0$ and on the other hand  that  ${\cal P}^{(n,T)}(\omega,F_{_T})$ is very close  to the random measures $\bar{\nu}^{(n)}(\xi(\omega),d\beta)$ of Theorem~\ref{thprincipal}(b) by taking $\usim{u}= \underline{u}\vee [u]$ and exactly equal to its continuous time counterpart in Theorem~\ref{thprincipal3}  if one sets  $\usim{u}=u$.   (This fact will be made more precise in Section~\ref{PMT}). 
 
Hence, the main step of the proof of the above theorems will be to study the rate of convergence of the sequence $({\cal P}^{(n,T)}(\omega,F_{_T}))_{n\ge0}$ to $\PE_\nu(F_{_T})$ for which the main result is given in Section~\ref{PMT} (see Proposition~\ref{proppnt}). 
In this way, we state in this section a series of preliminary lemmas.  In Lemma~\ref{lemme1}, we decompose the error between this new sequence $({\cal P}^{(n,T)}(\omega,F_{_T}))_{n\ge 1}$ and the target $\PE_\nu(F_{_T})$. In Lemma~\ref{lemmem1}, we recall a series of results on the
stability of diffusion processes and their genuine Euler scheme in finite horizon. Finally, in Lemma~\ref{lemme0}, we recall and extend  results of~\cite{LP2} about  the long-time behavior of the marginal Euler scheme.

For every $k\in\EN$, we define the ${\cal G}_k$-measurable random variable $\phi_{_F}(k)$ by 
\begin{equation}\label{phik}
\phi_{_F}(1)  = 0,\quad \phi_{_F}(k)= 
\int_{I_{k-1}} F_{_T}(\xi^{(\usim{u})})du\quad \textnormal{if $k\ge 2$.}
\end{equation}
where $I_{k}=[(k-1)T,kT)$. Please note that $\phi_{_F}(k)$ is $\fkt$-measurable.
\begin{lemme}\label{lemme1}
For every  $F$ satisfying $\cunf$ and $\cdeuxf$, we have
\begin{align*}
&\pnt(\omega,F_{_T})-\PE_{\nu}(F_{_T})=\frac{M_n}{nT}+\frac{\Theta_{n,1}+\Theta_{n,2}+\Theta_{n+1,3}}{nT}
\end{align*}
where

\smallskip 
\noindent $(M_n)_{n\ge1}$ is a $(\gn)$-martingale decomposed as follows : $\displaystyle{M_n=\sum_{i=1}^4 M_{n,i}}$ with
\begin{align*}
&\Delta M_{k,1}= \phi_{_F}(k)-\ES_{k-1}[\phi_{_F}(k)],& \Delta M_{k,2}= \ES_k[\phi_{_F}(k+1)]-\ES_{k-1}[\phi_{_F}(k+1)],\\
& \Delta M_{k,3}= \int_{I_k}\ES_{k-1}[F_{_T}(X^{({\usim{u}}),\xi_{{{\usim{u}}}}})]-f_{_F}(\xi_{{\usim{u}}})du,&\Delta M_{k,4}= -\int_{I_k}\psg \nabla g_{_F}(\xi_{{\usim{u}}}),\sigma(\xi_{{\usim{u}}}) dW_u\psd
\end{align*}
and $(\Theta_{n,1})$, $(\Theta_{n,2})$ and $(\Theta_{n,3})$ are $(\gn)$-adapted sequences defined for every $n\ge1$, by:
\begin{align*}
&\Theta_{n,1}=\sum_{k=1}^n\int_{I_k}\ES_{k-1}\left[F_{_T}({\xi}^{({\usim{u}})})-F_{_T}(X^{({\usim{u}}),\xi_{{{\usim{u}}}}})\right]du,\\
&\Theta_{n,2}=\sum_{k=1}^n \left(\int_{I_k}{\cal A}g_{_F}({\xi}_{{\usim{u}}})du-\Delta M_{k,4}\right),\\\
&\Theta_{n,3}=\left(\phi_{n}(F)-\ES_{n}(\phi_{n}(F))\right).
\end{align*}
\end{lemme}

\begin{proof}
With our newly defined notations, we have, for every $n\ge 1$, 
\[
\pnt(\omega,F_{_T})
=\frac{1}{nT} \sum_{k=1}^n \phi_{_F}(k+1).
\]
Now, for every $k\ge 1$, going twice backward through martingale increments, one checks that
 \[
  \phi_{_F}(k+1)= \Delta M^1_{k+1} +\Delta M^2_k +\ES_{k-1} (\phi_{_F}(k+1)).
 \]
Then, noting that $\displaystyle \ES_{k-1} (\phi_{_F}(k+1))=\int_{I_k} \ES_{k-1}(F_{_T}({\xi}^{({\usim{u}})}))du$, we introduce the approximation term $\Delta \Theta_{n,1}$ between the genuine Euler scheme $\xi$ and the true diffusion $X$  so that 
\[
\phi_{_F}(k+1)= \Delta M^1_{k+1} +\Delta M^2_k +\Delta \Theta_{k,1}+ \int_{I_k} \ES_{k-1}(F_{_T}(X^{({\usim{u}}),\xi_{{{\usim{u}}}}}))du.
\]
At this stage the Markov property applied to the original diffusion process yields
\[
 \ES_{k}(F_{_T}(X^{({\usim{u}},\xi_{{{\usim{u}}}})})) =  \ES_{k}\big( \ES_{ \usim{u}} F_{_T}(X^{({\usim{u}}),\xi_{{{\usim{u}}}}})\big)= \ES_{k}f_{_F}(\xi_{\usim{u}})=f_{_F}(\xi_{\usim{u}})
\]
since $\usim{u}\le u\le kT$. As a consequence, $\Delta M^3_k$ is a true ${\cal G}_k$-martingale increment and
\[
\phi_{_F}(k+1)= \Delta M^1_{k+1} +\Delta M^2_k +\Delta \Theta_{k,1}+\Delta M^3_k + \int_{I_k}  f_{_F}(\xi_{\usim{u}})du
\]
  On the other hand $f_{_F} = {\cal A}g_{_F} +\PE_{\nu}(F)$, so that
\begin{eqnarray*}
\int_{I_k}  f_{_F}(\xi_{\usim{u}})du-\PE_{\nu}(F)\, = \,\int_{I_k}   {\cal A}g_{_F} (\xi_{\usim{u}})du\,=\, \Delta \Theta_{n,2}+\Delta M^4_k.
\end{eqnarray*}
Finally, summing up all these terms yields
\begin{eqnarray*}
\pnt(\omega,F_{_T})-\PE_{\nu}(F) &=& \frac{1}{nT}\left(M^1_{n+1} +\sum_{i=2}^4M^i_n +\sum_{i=1}^2  \Theta_{n,i}\right) \\\\
&=&  \frac{1}{nT}\left(\sum_{i=1}^4 M^i_n +\sum_{i=1}^3\Theta_{n,i}\right)
\end{eqnarray*}
since $\Theta_{n+1,3}= M^1_{n+1}-M^1_n$.
 \end{proof}
 
\noindent \begin{Remarque} The term $\Theta_{n,1}$ sums up the error resulting from the approximation of $X^{({\usim{u}}),\xi_{{{\usim{u}}}}}$ by its Euler scheme (with decreasing step) $\xi_{\usim{u}+.}$. The term $\Theta_{n,2}$ is a residual approximation term as well: indeed, if we replace {\em mutatis mutandis} $\xi_{\usim{u}}$ by $X_u$,  It\^o's formula implies that
\[
g_{_F}(X_{(k+1)T})-g_{_F}(X_{kT})=\int_{I_k} {\cal A}g_{_F}(X_u)du +\int_{I_k}\psg \nabla g_{_F}(X_{u}),\sigma(X_{u}) dW_u\psd, 
\]
so that the resulting term would be, instead of $\Theta_{n,2}$,  $\frac{g_{_F}(X_{(n+1)T})-g_{_F}(X_{nT})}{nT} = O(1/n)$.
 \end{Remarque}
\begin{lemme}\label{lemmem1}  Let $p>0$ and  $T>0$. Assume that $b$ and $\sigma$ are Lipschitz continuous functions and that there exists  $\phi\!\in{\cal E}\!{\cal Q}(\ER^d)$ such that $|b|^2+\|\sigma\|^2\le C_{b,\sigma} \phi$ for a positive real constant $C_{b,\sigma}$. Then,

\medskip 
\noindent $(i)$  There exists a real constant $C_{p,T, b,\sigma}>0$, such that for every $u\ge 0$  and every finite ${\cal F}_u$-measurable random vector $\Xi$
$$
\ES[\sup_{t\in[0,T]} \phi^p(X_{t}^{(u),\Xi})\,|\, {\cal F}_u]\le C_{p,T, b,\sigma}\phi^p(\Xi)\quad\textnormal{and}\quad\ES[\sup_{t\in[0,T]} \phi^p(\xi_{u+t})\,|\, {\cal F}_u]\le C_{p,T, b,\sigma}\phi^p(\xi_{u}).
$$

\noindent $(ii)$  There exists a real constant $C_{p,T}>0$ such that, for every $u\ge0$,
$$
\ES[\sup_{t\in[0,T]}|\xi_{u+t}-X^{(u),\xi_u}_{t}|^p\,|\, {\cal F}_u]\le C_{p,T}(1+|\xi_u|^p)\gamma_{N(u)+1}^\frac{p}{2}.
$$

\noindent $(iii)$ There exists a real constant $C_{p}>0$ such that, for every $n\ge0$,
$$
\ES[\sup_{u\in[\Gamma_n,\Gamma_{n+1})} |\xi_u-\xi_{_{\Gamma_n}}|^p|{\cal F}_{\Gamma_n}]
\le C_{p} \phi^{\frac{p}{2}}(\xi_{_{\Gamma_n}})\gamma_{n+1}^{\frac{p}{2}}.
$$

\noindent $(iv)$  Let $p>2$. Then, there exists $C_{p,T,\delta}>0$ such that, for every $u\ge0$,

$$
\ES[\sup_{t\in[0,T]}|\xi_{u+t}-\xi_{\underline{u+t}}|^p\,|\, {{\cal F}_{\underline{u}}}]\le C_{p,T}\phi^{\frac{p}{2}}(\xi_{\underline{u}})\gamma_{N(u)+1}^{\frac{p}{2}-1}.
$$
\end{lemme}

\begin{proof} The  proofs follow the lines of their classical counterpart for the constant step Euler scheme of a diffusion (see $e.g.$ \cite{BOLE}, Theorem B.1.4 p.276 and the remark that follows). In particular, as concerns $(ii)$, the only thing to be checked  is that $(\xi_{u+t})_{t\ge 0}$ is the Euler scheme with decreasing step $\gamma^{(u)}$ of $X^{(u),\xi_u}$ where the step sequence $\gamma^{(u)}$ is defined by
\begin{equation}\label{gammau}
\gamma^{(u)}_1 = \Gamma_{N(u)+1}-u,\; \gamma^{(u)}_k= \gamma_{N(u)+k},\; k\ge 2.
\end{equation}
\end{proof}
\begin{lemme}\label{lemme0} Let $p>2$ and $a\in(0,1]$ such that $\mathbf{(S_{a,p})}$ holds and assume that $b$ and $\sigma$ are Lipschitz continuous functions. 

\smallskip
\noindent $(i)$ Let $g:\ER_+\rightarrow\ER_+$ be a nonincreasing function such that $\int_0^\infty g(u)du<+\infty$. Let $(\delta_k)$ be a nonincreasing sequence of positive numbers such that $\sum_{k\ge1}\delta_k<+\infty$. Then, $a.s.$, 
\begin{align}\label{01}
&\int_0^{+\infty} \ES[V^{p+a-1}(\xi_{\underline{u}})] g(u)du<+\infty \quad \textnormal{and}\quad\sum_{k\ge1}\delta_k \ES[V^{{p}+a-1}(\xi_{(k-1)T})]<+\infty.
\end{align}

\smallskip
\noindent $(ii)$ We have: 
\begin{equation}\label{lyapcontrol1}
\sup_{t\ge\Gamma_1}\frac{1}{t}\int_0^t V^{\frac{p}{2}+a-1}(\xi_{\un{s}})ds<+\infty\quad a.s.
\end{equation}
and
\begin{equation}\label{lyapcontrol2}
\sup_{n\ge1}\frac{1}{n}\sum_{k=1}^n V^{\frac{p}{2}+a-1}(\xi_{(k-1)T})<+\infty\quad a.s.
\end{equation}
In particular, the families of empirical measures $\big(\frac 1t \int_0^t\delta_{\xi_{\un{s}}}ds\big)_{t\ge1}$ and $\big(\frac 1n\sum_{k=1}^n{\delta_{\xi_{(k-1)T}}}\big)_{n\ge1}$ are $a.s.$ tight.

\smallskip
\noindent $(iii)$  Assume $\mathbf{(S^\nu_T)}$. Then, $a.s.$, for every continuous function $f$ such that $f(x)=o(V^{\frac{p}{2}+a-1}(x))$ as $|x|\rightarrow+\infty$, 
$$\frac{1}{t}\int_0^t f(\xi_{\un{s}})ds\xrn{t\nrn}\nu(f)\quad\textnormal{and}\quad\frac{1}{n}\sum_{k=1}^n f(\xi_{(k-1)T})\xrn{t\nrn}\nu(f).$$
\end{lemme}


\begin{proof}  $(i)$ First, note that 
$$\int_0^\infty V^{p+a-1}(\xi_{\underline{u}}) g({u})du=\sum_{n\ge1}\theta_n\gamma_nV^{p+a-1}(\xi_{\Gamma_{n-1}}),$$
where $\theta_n=\gamma_{n}^{-1}\int_{\Gamma_{n-1}}^{\Gamma_n}g(u)du$. 
Consequently, the first statement is simply a rewriting with  continuous time notations 
of  Lemma 4 of~\cite{LP2}. As concerns the second one, using Lemma~\ref{lemmem1}$(i)$ with $\phi=V$ and the exponent $^{p+a-1}$ yields  for every $k\ge1$ and  every  $u\in I_k$:
$$
\ES[V^{p+a-1}(\xi_{kT})]\le C_{p,a,T}\ES[V^{p+a-1}(\xi_{\underline{u}})].
$$
As a consequence, considering the integrable, nonincreasing, nonnegative function  $g= \sum_{k\ge 1} \mbox{\bf 1}_{I_{k-1}}\delta_k$ leads to
$$
\sum_{k\ge2}\delta_k \ES[V^{p+a-1}(\xi_{(k-1)T})]\le C_{p,a,T}\sum_{k\ge2}\int_{I_{k-1}}\ES[V^{p+a-1}(\xi_{\underline{u}})] g(u)du<+\infty
$$
owing to the previous statement.

\smallskip
\noindent $(ii)$ Set $r=\frac p2+a-1>0$ since $p>2$ and $a>0$. First, for every $n\ge1$ and  every $t\in[ \Gamma_n,\Gamma_{n+1})$,
$$
\frac{1}{t}\int_0^t V^{r}(\xi_{\underline{s}})ds\le \frac{\Gamma_{n+1}}{\Gamma_n}\frac{1}{\Gamma_{n+1}}\sum_{k=1}^{n+1} \gamma_k V^r(\xi_{_{\Gamma_{k-1}}})
\le \frac{2}{\Gamma_{n+1}}\sum_{k=1}^{n+1} \gamma_k V^r(\xi_{_{\Gamma_{k-1}}}),
$$
since $\gamma_n$ is nonincreasing. Now, owing to Proposition~\ref{funcpremiere1}, 

$$
\sup_{n\ge1}\frac{1}{\Gamma_n}\sum_{k=1}^n \gamma_k V^r(\xi_{_{\Gamma_{k-1}}})<+\infty\quad a.s.
$$
\noindent and~\eqref{lyapcontrol1} follows.

\noindent Let us deal now with~\eqref{lyapcontrol2}.
Given~\eqref{lyapcontrol1}, it is clear that~\eqref{lyapcontrol2} is equivalent to showing that for an increasing sequence $(t_k)$ such that
$t_0=0$, $\sup_{k\ge1} (t_k-t_{k-1})<+\infty$ and $t_k\rightarrow+\infty$,
\begin{equation}\label{116}
\sup_{n\ge1}\frac{1}{n}\sum_{k=1}^n \left((t_k-t_{k-1})V^r(\xi_{kT})-\int_{t_{k-1}}^{t_k}V^r(\xi_{\un{u}})du\right)<
+\infty\quad a.s.
\end{equation}
Setting $t_k=\Gamma_{N(kT)+1}$ for every $k\ge1$, this suggests to introduce the martingale defined by $N_0=0$ and for every $n\ge1$,
$$
N_n=\sum_{k=1}^n \frac{1}{k}\left(\int_{t_{k-1}}^{t_k}V^r(\xi_{kT})-V^r(\xi_{\un{u}})du-\ES_{t_{k-1}}\left[\int_{t_{k-1}}^{t_k}V^r(\xi_{kT})-V^r(\xi_{\un{u}})du\right]\right),\; 
$$ 
where  $\ES_{t_k}[\,.\,]:=\ES[\,.\,|{\cal F}_{t_{k}}].$  
Set $\varepsilon=\frac{p}{2r}$ so that $(1+\varepsilon)r= p+a-1$. Using that $\sup_{k\ge1} (t_k-t_{k-1})<+\infty$ and the elementary inequality
$|u+v|^{1+\varepsilon}\le 2^\varepsilon(u^{1+\varepsilon}+v^{1+\varepsilon})$ for  $u$, $v\ge 0$, 
\begin{align*}
\sum_{k\ge1}\frac{1}{k^{1+\varepsilon}} \ES \Big|&\int_{t_{k-1}}^{t_k}V^r(\xi_{kT})-V^r(\xi_{\un{u}})du\Big|^{1+\varepsilon}\\& \le C\sum_{k\ge1}\delta_k\ES[V^{r(1+\varepsilon)}(\xi_{kT})]+C\int_{0}^{+\infty}\ES[V^{r(1+\varepsilon)}(\xi_{\un{u}})] g(u)du,
\end{align*}
where $\delta_k= k^{-(1+\varepsilon)}$ and $g$ is the nonincreasing function defined by  $g(u)=k^{-(1+\varepsilon)}$ on $[t_{k-1},t_k)$. Thus, we deduce from~\eqref{01}  that,
\begin{align*}
\sum_{k\ge1}\frac{1}{k^{1+\varepsilon}} \ES \Big|\int_{t_{k-1}}^{t_k}V^r(\xi_{kT})-V^r(\xi_{\un{u}})du\Big|^{1+\varepsilon} <+\infty.
\end{align*}
It follows from the Chow Theorem (see $e.g.$~\cite{hall}) that $(N_n)$ $a.s.$ converges toward a finite random variable $N_{\infty}$ which in turn implies by  the Kronecker Lemma that 
$$
\frac{1}{n}\sum_{k=1}^n\left(\int_{t_{k-1}}^{t_k}V^r(\xi_{kT})-V^r(\xi_{\un{u}})du-\ES_{t_{k-1}}\left[\int_{t_{k-1}}^{t_k}V^r(\xi_{kT})-V^r(\xi_{\un{u}})du\right]\right)\xrightarrow{n\rightarrow+\infty}0\quad a.s.
$$ 
Then,  ~\eqref{116} will follow from
\begin{equation}\label{eq:118}
\sup_{n\ge1}\frac{1}{n}\sum_{k=1}^n \int_{t_{k-1}}^{t_k}\ES_{t_{k-1}}\left[V^r(\xi_{kT})-V^r(\xi_{\un{u}})\right]du<+\infty\quad a.s.
\end{equation}
In order to prove~\eqref{eq:118},  we need to inspect two cases for $r$:

\smallskip
\noindent Case  \fbox{$r\ge1$.} We decompose the increment $V^r(\xi_{kT})-V^r(\xi_{\un{u}})$ into  elementary increments, namely
 $$
 V^r(\xi_{kT})-V^r(\xi_{\un{u}})=V^r(\xi_{kT})-V^r(\xi_{\underline{kT}})+\sum_{\ell=N(\underline{u})+1}^{N(kT)}V^r(\xi_{_{\Gamma_\ell}})-V^r(\xi_{_{\Gamma_{\ell-1}}}).
 $$
Owing to the second order Taylor formula, we have for every $\ell\!\in \{N(\underline{u})+1,\ldots,N(kT)\}$:
 \begin{eqnarray*}
 V^r(\xi_{_{\Gamma_\ell}})- V^r(\xi_{_{\Gamma_{\ell-1}}})&=&\gamma_l \psg \nabla V^r,b\psd (\xi_{_{\Gamma_{\ell-1}}})+\psg \nabla V^r(\xi_{_{\Gamma_{\ell-1}}}),\sigma
 (\xi_{_{\Gamma_{\ell-1}}})(W_{\Gamma_{\ell}}-W_{\Gamma_{\ell-1}})\psd\\&&+\frac{1}{2} D^2 V^r(\theta_l)(\xi_{_{\Gamma_\ell}}-\xi_{_{\Gamma_{\ell-1}}})^{\otimes 2}\quad \mbox{where $\theta_l\in(\xi_{_{\Gamma_{\ell-1}}},\xi_{_{\Gamma_\ell}})$}.
 \end{eqnarray*}
 
  Note that a similar development holds for $V^r(\xi_{kT})-V^r(\xi_{\underline{kT}})$.
 Now, one checks that the fact that $V\!\in{\cal E}\!{\cal Q}(\ER^d)$ implies that $\|D^2 V^r\|\le C_{_V}V^{r-1}$ and that 
 $\sqrt V$ is a Lipschitz continuous function with Lipschitz constant  $[\sqrt{V}]_1$. Consequently
 \begin{eqnarray*}
 |D^2 V(\theta_{\ell})(\xi_{_{\Gamma_\ell}}-\xi_{_{\Gamma_{\ell-1}}})^{\otimes 2}|&\le& C_{_V} \big(\sqrt{V(\xi_{_{\Gamma_{\ell-1}}})}+[\sqrt{V}]_1 |\xi_{_{\Gamma_\ell}}-\xi_{_{\Gamma_{\ell-1}}}|\big)^{2(r-1)}|\xi_{_{\Gamma_\ell}}-\xi_{_{\Gamma_{\ell-1}}} |^2\\
 &\le &C_{_{r,V}} V^{r-1}(\xi_{_{\Gamma_{\ell-1}}})|\xi_{_{\Gamma_\ell}}-\xi_{_{\Gamma_{\ell-1}}} |^2+ C|\xi_{_{\Gamma_\ell}}-\xi_{_{\Gamma_{\ell-1}}}|^{2r},
 \end{eqnarray*}
 where we used  in the second inequality the standard control $|u+v|^s\le 2^{s-1}(|u|^s+|v|^s)$.
Then, summing over $\ell$ and using that $\psg\nabla V,b\psd \le \beta$ owing to $ \mathbf{(S_{a,p})}(ii)$, we deduce that 
 \begin{eqnarray*}
V^r(\xi_{kT})-V^r(\xi_{\un{u}})&\le & \beta (kT-\un{u})+\int_{\un{u}}^{kT} \psg \nabla V^r(\xi_{\un{{v}}}),\sigma
 (\xi_{\un{{v}}})dW_{v}\psd\\
 &&+C_{_V}\int_{\un{u}}^{kT} V^{r-1}(\xi_{\un{{v}}})|\xi_{\bar{{v}}\wedge kT}-\xi_{\un{{v}}}|^2+ |\xi_{\bar{{v}}\wedge kT}-\xi_{\un{{v}}}|^{2r}\frac{d{v}}{\gamma_{{N(v)+1}}}
 \end{eqnarray*}
where $\bar{{v}}=\Gamma_{N({v})+1}$. By $\mathbf{(S_{a,p})}(i)$, we can use Lemma~\ref{lemmem1}$(iii)$ with $\phi=V^a$ and $p=s$  to obtain for every $s>0$, 
\begin{equation}\label{IneqTechXi}
\mathbb{E}_{\un{{v}}} [|\xi_{\bar{{v}}\wedge kT}-\xi_{\un{{v}}}|^s ]\le C_{s} V^{\frac{as}{2}}(\xi_{\un{{v}}})\gamma_{N({v})+1}^{s/2}.
\end{equation}
Applying successively the above inequality with $s=2$ and $s=2r\ge 2$ and using the chain rule for conditional expectations show that,
 \begin{eqnarray*}
\mathbb{E}_{t_{k-1}}[V^r(\xi_{kT})-V^r(\xi_{\un{u}})]&\le& \beta (T+\|\gamma\|_{\infty})+\ES_{{t_{k-1}}}\left[\int_{\un{u}}^{kT} V^{r+a-1}(\xi_{\un{v}})dv\right]\\
&\le& C_{T,\beta, \|\gamma\|_{\infty}}\left(1+\ES_{t_{k-1}}\left[\int_{t_{k-1}}^{t_k} V^{r+a-1}(\xi_{\un{u}})du\right]\right) 
 \end{eqnarray*}
for some real constant $C_{T,\beta \|\gamma\|_{\infty}}$.  As a consequence, 
\begin{equation*}
\sup_{n\ge1}\frac{1}{n}\sum_{k=1}^n \int_{t_{k-1}}^{t_k}\ES_{t_{k-1}}\left[V^r(\xi_{kT})-V^r(\xi_{\un{u}})\right]du
\le  C\left(1+\sup_{n\ge1}\frac{1}{n}\sum_{k=1}^n\ES_{t_{k-1}}\left[\int_{t_{k-1}}^{t_k} V^{r+a-1}(\xi_{\un{u}})du\right]\right).
\end{equation*}
Let $\varepsilon\!\in(0,\frac{p+a-1}{r+a-1})$ (note that $\frac{p+a-1}{ r+a-1}= \frac{\frac p2 -(a-1)}{\frac p2+2(a-1)}>0$ since $p>2$ and $0<a\le 1$). Hence $(1+\varepsilon)(r+a-1)\le p+a-1$ and by Lemma~\ref{lemmem1}$(i)$ and~\eqref{01}, one checks  that
\begin{align*}
\sum_{k=1}^{+\infty}\frac{1}{k^{1+\varepsilon}}\ES_{t_{k-1}}\Big|\ES_{t_{k-1}}\Big[\int_{t_{k-1}}^{t_k} &V^{r+a-1}(\xi_{\un{u}})du\Big]-\int_{t_{k-1}}^{t_k} V^{r+a-1}(\xi_{\un{u}})du\Big|^{1+\varepsilon}\\&\le C \sum_{k=1}^{+\infty}\frac{1}{k^{1+\varepsilon}}\int_{t_{k-1}}^{t_k}\ES_{t_{k-1}}\left[ V^{p+a-1}(\xi_{\un{u}})\right]du<+\infty\quad a.s.
\end{align*}
by the first part of the lemma. Then, one derives  using a  martingale argument  based on~\eqref{lyapcontrol1}, the Chow Theorem  and the Kronecker Lemma  that
\begin{equation*}
\sup_{n\ge1}\frac{1}{n}\sum_{k=1}^n\ES_{t_{k-1}}\left[\int_{t_{k-1}}^{t_k} V^{r+a-1}(\xi_{\un{u}})du\right]<+\infty.
\end{equation*}
 
\noindent Case \fbox{$0<r\le 1$.} In that case, we just use that $D^2V^r$ is bounded so that we just have  to use~(\ref{IneqTechXi}) with $s=2$ (since $a<p+a-1$).  This concludes the proof of $(ii)$.

\smallskip
\noindent $(iii)$ The fact that $a.s.$, $\displaystyle{\frac{1}{t}\int_0^t f(\xi_{\un{s}})ds\xrightarrow{t\rightarrow+\infty}\nu(f)}$ is but the statement of Proposition~\ref{funcpremiere1} with continuous time notations. Now, let us show that $a.s.$, for every continuous function $f$ such that $f=o(V^{\frac{p}{2}+a-1})$,
\begin{equation}\label{112}
\frac{1}{n}\sum_{k=1}^n f(\xi_{(k-1)T})\xrn{n\nrn}\nu(f).
\end{equation}
First, taking advantage of~\eqref{lyapcontrol2}, standard weak convergence arguments based on  uniform integrability  show that it is enough to prove that, $a.s.$,~\eqref{112}
holds  for every bounded continuous function $f$. Then, using that weak convergence on $\ER^d$ can be characterized along a countable subset ${\cal S}$ of Lipschitz bounded continuous functions $f$, the problem amounts to showing   that for every  Lipschitz bounded continuous function $f:\ER^d\rightarrow\ER$
\begin{equation}\label{114}
\frac{1}{n}\sum_{k=1}^n f(\xi_{(k-1)T})\xrn{n\nrn}\nu(f)\quad a.s.
\end{equation}
Owing to $\mathbf{(S^\nu_T)}$, our strategy here will be to show that almost any    limiting distribution of the empirical measures is invariant since it leaves the transition operator $P_{_T}$ invariant. As a first step, we first derive from a standard martingale argument that 
\begin{equation}\label{114}
\frac{1}{n}\sum_{k=2}^n f(\xi_{(k-1)T})-\ES_{k-2}[f(\xi_{(k-1)T})] \xrn{n\nrn}0\quad a.s.
\end{equation}
Now, we remark that
\begin{eqnarray}
\ES_{k-2}[f(\xi_{(k-1)T})]&=& P_{_T} f(\xi_{(k-2)T})+R_{k-2}(\xi_{(k-2)T})\qquad\label{1145}\\
\quad \textnormal{with }\hskip 4 cm  R_{k}(x)&=& \ES[f(\xi_T^{x,\gamma^{(kT)}})-f(X_T^{x})],\hskip 3,5 cm
\end{eqnarray}
where $\xi^{x,\gamma^{(kT)}}$ denotes the genuine Euler scheme starting from $x$ with step sequence $\gamma^{(kT)}$ defined by~\eqref{gammau}.
Since $f$ is bounded Lipschitz,
\begin{equation*}
R_{k}(x)\le C \ES[|\xi_T^{x,\gamma^{(kT)}}-X_T^{x}|]\mbox{\bf 1}_{\{|x|\le M\}}+2\|f\|_\infty \mbox{\bf 1}_{\{|x|>M\}}
\le C_M\sqrt{\gamma_{N(k)}}+ 2\|f\|_\infty \mbox{\bf 1}_{\{|x|>M\}},
\end{equation*}
where in the second inequality, we used Lemma~\ref{lemmem1}$(ii)$ with $p=1$. Thus, since $\gamma_{N(k)}\xrn{k\nrn}0$, it follows from~\eqref{1145} that, 
for every $M>0$,
\begin{equation*}
\limsup_{n\rightarrow+\infty}\frac{1}{n}\sum_{k=2}^n \left(\ES_{k-2}[f(\xi_{(k-1)T})]-P_{_T}[f(\xi_{(k-2)T})]\right)\le 
C\limsup_{n\rightarrow+\infty}\frac{1}{n}\sum_{k=2}^n\mbox{\bf 1}_{B(0,M)^c}(\xi_{(k-1)T})\quad a.s.
\end{equation*}
Then, it follows from~\eqref{114} and from the a.s. tightness of $\displaystyle \big(\frac{1}{n}\sum_{k=1}^n\delta_{\xi_{(k-1)T}}\big )_{n\ge1}$ that, $a.s.$,
\begin{equation*}
\frac{1}{n}\sum_{k=1}^n \left(f(\xi_{(k-1)T})-P_{_T}f(\xi_{(k-1)T})\right)=\frac{1}{n}\sum_{k=2}^n \left(f(\xi_{(k-1)T})-P_{_T}f(\xi_{(k-2)T})\right)+O(\frac{1}{n})\xrn{n\nrn}0.
\end{equation*}
Now, since $f$ and $P_{_T} f$ are bounded continuous, it follows that, $a.s.$, for every weak limit $\nu_\infty(\omega,dx)$ of the tight sequence $(n^{-1}\sum_{k=1}^n\delta_{\xi_{(k-1)T}})_{n\ge1}$, $\nu_\infty(\omega,f)=\nu_\infty(\omega,P_{_T} f)$ for every $f\in{\cal S}$. This implies that $\nu_\infty(\omega,dx)$ is an invariant distribution for $P_{_T}$  and one concludes the proof  by $\mathbf{(S^\nu_T)}$.
\end{proof} 

\section{Rate of convergence for the martingale component}\label{sectionmartingale}
This section is devoted to the study of the rate of convergence of the martingale $(M_n)$ defined in Lemma~\ref{lemme1}.
The main result of this section is Proposition~\ref{propTCL} where we obtain a $CLT$ for this martingale. On the way to this result, the main difficulty is to study the asymptotic behavior of the previsible bracket of this sum of four dependent martingales. First, 
we decompose the martingale increment $\Delta M_n$ as follows:
\begin{align*}
\Delta M_n=\ES_n[\bar{A}_{n+1}+\bar{B}_n]-\ES_{n-1}[\bar{A}_{n+1}+\bar{B}_n],
\end{align*}
where $(\bar{A}_n)$ is a $(\gn)$-adapted sequence defined for every $n\ge1$ by:
\begin{align*}
&\bar{A}_{n}=\phi_{n-1}+\phi_{n}-\int_{I_{n-1}} f_{_F}(\xi_{\usim{u}})du=\int_{(n-3)T}^{(n-1)T} F_{_T}({\xi}^{(\usim{u})})du-\int_{I_{n-1}} f_{_F}(\xi_{\usim{u}})du,
\end{align*}
and $\bar{B}_n=\Delta M_{n,4}$. Keep in mind that $\ES_{n-1}[\bar{B}_n]=0$. In the following lemma,
we set
$$
Z^k:=X^{(kT),\xi_{kT}}\quad\forall k\ge1,
$$ 
where, following the notation introduced in~\eqref{equationshiftee}, $X^{(kT),\xi_{kT}}$ denotes the unique solution to $dY_t=b(Y_t)dt+\sigma(Y_t)dW^{(kT)}$ starting from $\xi_{kT}$.
\begin{lemme}\label{lemme2} Assume $b$ and $\sigma$  are Lipschitz continuous functions 
satisfying  $\mathbf{(S_{a,p})}$ with an essentially quadratic Lyapunov function  $V$ and parameters $a\in(0,1]$ and $p>2$.
Let $F:\mathbb{D}_{uc}([0,T],\ER^d)\to \ER$ denote a functional satisfying $\mathbf{(C_F^1)}$ and $\mathbf{(C_F^2)}$. Then,
\begin{equation}\label{Eq51}
\frac{1}{n}\sum_{k=2}^n\ES_{k-1}[(\Delta M_k)^2]-\ES_{k-2}[(\Delta M_{k})^2]\xrn{n\nrn}0\quad a.s.
\end{equation}
and 
\begin{equation}\label{Eq52}
\frac{1}{n}\sum_{k=2}^n\left(\ES_{k-2}[(\Delta M_k)^2]-\left(\ES_{k-2}[(\ES_k C_{k+1})^2]-\ES_{k-2}[(\ES_{k-1} C_{k+1})^2]\right)\right)\xrn{n\nrn}0\quad a.s.
\end{equation}
where $C_{k+1}=A_{k+1}+B_k$ with
\begin{align*}
&A_{k+1}=\int_{0}^{2T}F_{_T}(Z^{k-2}_{u+.})du-\int_T^{2T} f_{_F}(Z^{k-2}_u)du\\
 \hbox{and }\hskip 3,1 cm  &B_k=-\int_0^T \psg \nabla g_{_F}(Z^{k-2}_u),\sigma(Z^{k-2}_u)dW_u^{(k-2)T}\psd.\hskip 3,3 cm 
\end{align*}

\end{lemme}
\begin{proof}  We consider  the $({\cal G}_{n-1})$-martingale $(N_n)$ defined by: 
$$
N_n:=\sum_{k=2}^{n}\frac{1}{k}\left(\ES_{k-1}[(\Delta M_k)^2]-\ES_{k-2}[(\Delta M_{k})^2]\right).
$$
Let $\varepsilon>0$. Using Jensen's inequality, we have 
\begin{eqnarray*}
\sum_{k\ge2}\ES_{k-2}\left|\Delta N_k\right|^{1+\varepsilon}
\le C \sum_{k\ge 2}\frac{1}{k^{1+\varepsilon}}\ES_{k-2}\left|\Delta M_k\right|^{2(1+\varepsilon)}\le C\sum_{k\ge 2}\frac{1}{k^{1+\varepsilon}} \ES_{k-2}\left|\bar{A}_{k+1}+\bar{B}_k\right|^{2(1+\varepsilon)}.
\end{eqnarray*}
Using successively conditional Burkh\"older-Davis-Gundy and Jensen inequalities and $\cunf$, we have 
\begin{eqnarray}\nonumber
\ES_{k-2}\left|\bar{A}_{k+1}+\bar{B}_k\right|^{2(1+\varepsilon)}&\le&3^{1+2\varepsilon}\left((2\|F\|_{\infty} T)^{2(1+\varepsilon)}+(\|F\|_{\infty} T)^{2(1+\varepsilon)}\right.\\
&&\left.+T^\varepsilon\int_{I_k}\ES_{k-2}\left[|\nabla g_{_F}(\xi_{\usim{u}})|^{2(1+\varepsilon)}\|\sigma(\xi_{\usim{u}})\|^{2(1+\varepsilon)}\right]du\right).\label{bk1}
\end{eqnarray}
Now, since $\nabla g_{_F}$ is bounded and $\|\sigma\|^2\le CV^a$,
\begin{eqnarray}\nonumber
\ES_{k-2}\left[|\nabla g_{_F}(\xi_{\usim{u}})|^{2(1+\varepsilon)}\|\sigma(\xi_{\usim{u}})\|^{2(1+\varepsilon)}\right]&\le &C\ES_{k-2}[V^{a(1+\varepsilon)}(\xi_{\usim{u}})] \\
&\le& C (1+ \bar{G}_{k-2,a(1+\varepsilon)}({\xi}_{(k-2)T}) )   \label{bk2}
\end{eqnarray}
where $\displaystyle{\bar{G}_{k,p}(x)=\ES[\sup_{t\in [0,T]}V^p(\xi^{x,\gamma^{(k)}}_t)].}$ By Lemma~\ref{lemmem1}(i) applied with $\phi=V$ and $p={a(1+\varepsilon)}$ with $\varepsilon\!\in(0, \frac{p-1}{a})$. it follows that for every $k\ge2$,
\begin{equation}\label{controlemk}
\ES_{k-2}\left|\Delta M_k\right|^{2(1+\varepsilon)}\le C_{F,\varepsilon,T}V^{a(1+\varepsilon)}(\xi_{(k-2)T}).
\end{equation}
Then, we deduce from Lemma~\ref{lemme0} applied with $\delta_k=k^{-(1+\varepsilon)}$ that
$$
\sum_{k\ge2}\ES_{k-2}\left|\Delta N_k\right|^{1+\varepsilon}\le \sum_{k\ge2}\frac{1}{k^{1+\varepsilon}} V^{a(1+\varepsilon)}(\xi_{(k-2)T})<+\infty \quad a.s.
$$
since $a(1+\varepsilon)<p+a-1$. Finally, using the Chow theorem, it follows  that $(N_n)$ is an $a.s.$ convergent martingale and the result follows from the Kronecker lemma.

\medskip
\noindent \textit{(ii)} Set $\bar{C}_{k}=\bar{A}_{k}+\bar{B}_{k-1}$. We have $\Delta M_k=\ES_k[\bar{C}_{k+1}]-\ES_{k-1}[\bar{C}_{k+1}]$ so that
$$
\ES_{k-2}[(\Delta M_k)^2]=\ES_{k-2}[(\ES_k\bar{C}_{k+1})^2]-\ES_{k-2}[(\ES_{k-1}\bar{C}_{k+1})^2].
$$
Thus, it is enough to show that
\begin{align}\label{part1}
&\frac{1}{n}\sum_{k=2}^n\ES_{k-2}[(\ES_k\bar{C}_{k+1})^2]-\ES_{k-2}[(\ES_k{C}_{k+1})^2]\xrn{n\nrn}0\quad a.s.\\
\textnormal{and }\qquad\qquad&\frac{1}{n}\sum_{k=2}^n\ES_{k-2}[(\ES_{k-1}\bar{C}_{k+1})^2]-\ES_{k-2}[(\ES_{k-1}{C}_{k+1})^2]\xrn{n\nrn}0\quad a.s.\qquad\label{part2}
\end{align}
Let us focus on~\eqref{part1}. Set $q=\frac{p}{p-1}$. Using conditional H\"older and Jensen inequalities, we obtain:
\begin{align}
|\ES_{k-2}[(\ES_k\bar{C}_{k+1})^2] - &\,\ES_{k-2}[(\ES_k{C}_{k+1})^2]|=|\ES_{k-2}\left[\ES_k(\bar{C}_{k+1}-C_{k+1})
\ES_k (\bar{C}_{k+1}+C_{k+1})\right]|\nonumber\\
\le &\;\ES_{k-2}\left[(\ES_k|\bar{A}_{k+1}-A_{k+1}|^p)^\frac{1}{p}
(\ES_k |\bar{C}_{k+1}+C_{k+1}|^q)^\frac{1}{q}\right]\nonumber\\
&+\ES_{k-2}\left[(\ES_k(\bar{B}_{k+1}-B_{k+1})^2)^\frac{1}{2}
(\ES_k (\bar{C}_{k+1}+C_{k+1})^2)^\frac{1}{2}\right]\nonumber\\
\le&\left(\ES_{k-2}|\bar{A}_{k+1}-A_{k+1}|^p\right)^\frac{1}{p}\left(\ES_{k-2}|\bar{C}_{k+1}+C_{k+1}|^q\right)^\frac{1}{q}\label{ABC2}\\
&+\left(\ES_{k-2}|\bar{B}_{k}-B_{k}|^2\right)^\frac{1}{2}\left(\ES_{k-2}(\bar{C}_{k+1}+C_{k+1})^2\right)^\frac{1}{2}.\label{ABC}
\end{align}
Let us inspect successively the three terms involved in~\eqref{ABC}.\\
Now set $\displaystyle{{G}_{p}(x)=\ES[\sup_{t\in [0,T]}V^p(X^{x}_t)]}$. Still using Lemma~\ref{lemmem1}$(i)$, we show (like previously for~(\ref{controlemk}))  that, for every $k\ge 2$ and $r\ge2$, 
\begin{equation}\label{part11}
\ES_{k-2}|\bar{C}_{k+1}+C_{k+1}|^r\le C\left(1+\bar{G}_{k-2,\frac{r}{2}}({\xi}_{(k-2)T})+G_{\frac{r}{2}}(\xi_{(k-2)T})\right)\le C V^{\frac{r}{2}}(\xi_{(k-2)T}).
\end{equation}
On the other hand since $F$ and $f_{_F}$ are bounded Lipschitz continuous functions,
\begin{align*}
\ES_{k-2}[&|\bar{A}_{k+1}-A_{k+1}|^p]\le C\left(1\wedge\ES_{k-2}\left[\sup_{v\in [(k-2)T,(k+1)T]}|\xi_{\usim{v}}-Z^{k-2}_v|^p\right]\right)\\
&\le C\left(\ES_{k-2}\left[\sup_{v\in [(k-2)T,(k+1)T]}|\xi_{\usim{v}}-\xi_{v}|^p\right]+\ES_{k-2}\left[\sup_{v\in [(k-2)T,(k+1)T]}|\xi_{v}-Z^{k-2}_v|^p\right]\right)\wedge 1.
\end{align*}
Then, owing to the Markov property,
\begin{align*}
&\ES_{k-2}[|\bar{A}_{k+1}-A_{k+1}|^p]\le C \left[\left( {\cal H}_{k-2,3T,p}(\xi_{(k-2)T})+ {\cal K}_{k-2,3T,p}(\xi_{(k-2)T})\right)\wedge 1\right]\quad\textnormal{with,}\\
&{\cal H}_{k,T,p}(x)=\ES[ \sup_{v\in[0,T]}|\xi_{\usim{v}}^{x,\gamma^{(kT)}}-\xi_v^{x,\gamma^{(kT)}}|^p]\quad
\textnormal{and }\quad {\cal K}_{k,T,p}(x)=\ES[ \sup_{v\in[0,T]}|\xi_v^{x,\gamma^{(kT)}}-X^x_v|^p],
\end{align*}
where $(\xi_v^{x,\gamma^{(kT)}})_{v\ge 0}$ denotes the Euler scheme of $X^x$ with step sequence $\gamma^{(kT)}$ as defined by~(\ref{gammau}). 
Now, using that for every $v\in[0,T]$,
\begin{align*}
|\xi_{\usim{v}}^{x,\gamma^{(kT)}}-\xi_v^{x,\gamma^{(kT)}}|^p&\le 2^{p-1}\left(|\xi_{v}^{x,\gamma^{(kT)}}-\xi_{\underline{v}}^{x,\gamma^{(kT)}}|^p+|\xi_{\usim{v}}^{x,\gamma^{(kT)}}-\xi_{\underline{v}}^{x,\gamma^{(kT)}}|^p\right)\\
&\le 2^p\sup_{v\in[0,T]}|\xi_v^{x,\gamma^{(kT)}}-\xi_{\underline{v}}^{x,\gamma^{(kT)}}|^p,
\end{align*}
it follows from Lemma~\ref{lemmem1}$(iv)$ that 
\begin{equation}\label{eq:calH}
{\cal H}_{k,T,p}(x)\le C\gamma_{N(kT)}^{\frac{p}{2}-1} V^{\frac{ap}{2}}(x),
\end{equation}
and by  Lemma~\ref{lemmem1}$(ii)$,
\begin{equation}\label{eq:calK}
{\cal K}_{k,T,p}(x)\le  C(1+|x|^p)\gamma_{N(kT)}^\frac{p}{2},
\end{equation}
so that, for every $M>0$, 
\begin{align}
\ES_{k-2}[|\bar{A}_{k+1}-A_{k+1}|^p] \le & \;C\gamma_{N((k-2)T)}^{\frac{p}{2}-1}(1+V^{\frac{ap}{2}}(\xi_{(k-2)T})+|\xi_{(k-2)T}|^p)\mbox{\bf 1}_{\{|\xi_{(k-2)T}|\le M\}}\nonumber\\
\label{330} &+C\mbox{\bf 1}_{\{|\xi_{(k-2)T|}>M\}}.
\end{align}
Finally, we have 
\begin{align*}
&\ES_{k-2}[|\bar{B}_{k}-B_{k}|^2] = \ES_{k-2}\left[\int_{I_k}|\sigma^*\nabla g_{_F}(\xi_{\usim{u}})-\sigma^*\nabla g_{_F}(Z^{k-2}_u)|^2 du\right],
\end{align*}

On the one hand   $\nabla g_{_F}$ and $\sigma$  being both Lipschitz continuous and $\nabla g_{_F}$ being bounded, we have for every $x,y\in\ER^d$,
\begin{equation}\label{25}
|\sigma^*\nabla g_{_F}(x)-\sigma^* \nabla g_{_F}(y)|^2\le C(1+ \|\sigma(y)\|^2)|x-y|^2.
\end{equation}
As a consequence, using Schwarz inequality and Assumption $\mathbf{(S_{a,p})}(i)$, it follows that
\begin{align*}
\ES_{k-2}[|\bar{B}_{k}&-B_{k}|^2]\\
 &\le C\left(\ES_{k-2}\left[1+\sup_{u\in[(k-1)T,kT]}V^{2a}(Z^{k-2}_{u})\right]\right)^\frac{1}{2}\left(\ES_{k-2}\sup_{u\in[(k-1)T,kT]}|\xi_{\usim{u}}-Z^{k-2}_u|^4\right)^\frac{1}{2}.
\end{align*}
Owing to Lemma~\ref{lemmem1}$(i)$, it follows that
\[
\ES_{k-2}[|\bar{B}_{k}-B_{k}|^2]
\le C\,V^{a}(\xi_{(k-2)T})({\cal H}_{k-2,2T,4}(\xi_{(k-2)T})+{\cal K}_{k-2,2T,4}(\xi_{(k-2)T}))^\frac{1}{2}. 
\]
and by~\eqref{eq:calH} and~\eqref{eq:calK} that,
\begin{align}
\ES_{k-2}[|\bar{B}_{k}-B_{k}|^2]
&\le C\,V^{a}(\xi_{(k-2)T})\left(\sqrt{\gamma_{N((k-2)T)}} V^{a}(\xi_{(k-2)T})+\gamma_{N((k-2)T)}(1+|\xi_{(k-2)T}|^2)\right)\nonumber\\
&\le  C'\,\sqrt{\gamma_{N((k-2)T)}}(1+V^{2a}(\xi_{(k-2)T})+|\xi_{(k-2)T}|^{2(1+a)})\label{331},
\end{align}
where we used in the last inequality that $V(x)\le C(1+|x|^2)$. On the other hand since 
\begin{equation*}
|\sigma^*\nabla g_{_F}(x)-\sigma^* \nabla g_{_F}(y)|^2\le C(\|\sigma(x)\|^2 +\|\sigma(y)\|^2),
\end{equation*}
we  deduce likewise from $\sapi$ and  Lemma~\ref{lemmem1}$(i)$ that
\begin{align}\label{332}
\ES_{k-2}|\bar{B}_{k}&-B_{k}|^2\le C V^a(\xi_{(k-2)T}).
\end{align}
Thus, plugging the inequalities obtained in~\eqref{part11},~\eqref{330},~\eqref{331} and~\eqref{332} into~\eqref{ABC} and~\eqref{ABC2} yields for every $M>0$,
$$
|\ES_{k-2}[(\ES_k\bar{C}_k)^2]-\ES_{k-2}[(\ES_k{C}_k)^2]|\le C_M\gamma_{N((k-2)T)}^{\frac 14\wedge(\frac{1}{2}-\frac 1p)}\mbox{\bf 1}_{\{|\xi_{(k-2)T}|\le M\}}+ C V^{a}(\xi_{(k-2)T})\mbox{\bf 1}_{\{|\xi_{(k-2)T}|> M\}}.
$$
Since $\gamma_{N(kT)}\to 0$ as $k\to\infty$ and $p>2$, it follows that $a.s.$, for every $M>0$, 
$$
\limsup_{n\rightarrow+\infty}\frac{1}{n}\sum_{k=2}^n\left|\ES_{k-2}[(\ES_k\bar{C}_k)^2]-\ES_{k-2}[(\ES_k{C}_k)^2]\right|\le C \limsup_{n\rightarrow+\infty}\frac{1}{n}\sum_{k=2}^n V^{a}(\xi_{(k-2)T})\mbox{\bf 1}_{\{|\xi_{(k-2)T}|> M\}}.
$$
Since $p>2$,  there exists $\varepsilon>0$ such that $a(1+\varepsilon)<\frac p2+a-1$. Hence, it follows from~\eqref{lyapcontrol2}, that
$$\sup_{n\ge1}\frac{1}{n}\sum_{k=2}^n V^{a(1+\varepsilon)}(\xi_{(k-2)T})<+\infty\quad a.s.$$
Then, we deduce by a standard uniform integrability argument that 
$$
\limsup_{M\rightarrow+\infty}\limsup_{n\rightarrow+\infty}\frac{1}{n}\sum_{k=2}^n V^{a}(\xi_{(k-2)T})\mbox{\bf 1}_{\{|\xi_{(k-2)T}|> M\}}=0\quad a.s.
$$
This completes the proof of~\eqref{part1}. The proof of~\eqref{part2} is similar and the details are left to the reader.
\end{proof}
\begin{lemme}\label{lemme3} Assume that $b$ and $\sigma$ are Lipschitz continuous functions.\\ 
\noindent $(i)$ For every $k\ge 1$,
\begin{equation}
\ES_{k-2}[(\ES_k C_{k+1})^2]-\ES_{k-2}[(\ES_{k-1} C_{k+1})^2]=\Psi({\xi}_{(k-2)T})\quad\textnormal{where,}
\end{equation}
$$
\Psi(x)=\ES\left[\Big(\ES\Big(A_{_{2T}}^x\,|\,{\cal F}_{_{2T}}\Big)-\ES\Big(A_{_T}^x\,|\,{\cal F}_{_T}\Big)-\int_T^{2T}\sigma^*\nabla g_F(X_u^x)dW_u\Big)^2\right],
$$
with $A^x_t:=\int_0^t \left(F_{_T}(X^x_{u+.})-f_{_F}(X^x_u)\right) du$, $t\ge 0$.

\smallskip
\noindent $(ii)$ If $\cunf$ holds, $\Psi$ is a continuous function on $\ER^d$. As a consequence, if moreover $\cdeuxf$, $\mathbf{(S^\nu_T)}$ and $\sap$ hold for $a\in(0,1]$ and $p>2$,
\begin{equation}\label{36}
\frac{1}{n}\sum_{k=2}^n\ES_{k-2}[(\Delta M_k)^2]\xrn{n\nrn}\sigma^2_{_F}=\int \Psi(x)\nu(dx)\quad a.s.
\end{equation}
\end{lemme}
\begin{proof} $(i)$ Let $\Lambda$ be a bounded (or nonnegative) Borel functional defined on ${\cal C}(\ER_+,\ER^d)$. Since pathwise uniqueness holds for $SDE$~\eqref{sde} ($b$ and $\sigma$ being Lipschitz continuous), there exists a measurable function $h:\ER^d\times{\cal C}(\ER_+,\ER^\ell)\rightarrow{\cal C}(\ER_+,\ER^d)$ such that $a.s.$, for every $k\ge2$, $Z^{k-2}=X^{((k-2)T),\xi_{(k-2)T}}=h(\xi_{(k-2)T},W^{((k-2)T)})$
(see $e.g.$~\cite{karatzas}, Corollary 3.23). Then, using that $\xi_{(k-2)T}$ is ${\cal G}_{k-2}= {\cal F}_{(k-2)T}$-measurable, that the Brownian motion $W^{((k-2)T)}$ is independent of $ {\cal F}_{(k-2)T}$ and that, ${\cal F}_{kT}={\cal F}_{(k-2)T}\vee {\cal F}^{W^{((k-2)T)}}_{{2T}}$, one derives that
\begin{eqnarray*}
\ES_{k}\Big(\Lambda(X^{((k-2)T),\xi_{(k-2)T}})\Big)&=& \ES\Big(\Lambda(X^{((k-2)T),\xi_{(k-2)T}})\,|\, {\cal F}_{2T}^{W^{((k-2)T)}}\Big)\\
&=&\ES\Big(\Lambda(X^{((k-2)T),x})\,|\, {\cal F}_{2T}^{W^{((k-2)T)}}\Big)_{|x=\xi_{(k-2)T}}.
\end{eqnarray*}
Using again the representation with function $h$ (or the fact that strong uniqueness implies weak uniqueness), one observes  that the spatial process $\displaystyle \left(\ES\Big(\Lambda(X^{((k-2)T),x})\,|\, {\cal F}_{2T}^{W^{((k-2)T)}}\Big)\right)_{x\in \ER^d}$ has the same distribution as $\displaystyle \left(\ES\Big(\Lambda(X^{x})\,|\, {\cal F}_{2T}^{W}\Big)\right)_{x\in \ER^d}$ where $(X^{x}_t)_{t\ge 0, x\in \ER^d}$ is the flow of $SDE$~(\ref{sde}) at tme $0$. Consequently,
\[
\ES_{k-2}\left(\ES_{k}\Big(\Lambda(X^{((k-2)T),\xi_{(k-2)T}})\Big)^2\right)= \left[\ES_x\left(\ES\Big(\Lambda(X^{x})\,|\, {\cal F}_{2T}^{W}\Big)\right)^2\right]_{x= \xi_{(k-2)T}}.
\]
Similar arguments show that 
\[
\ES_{k-2}\left(\ES_{k-1}\Big(\Lambda(X^{((k-2)T),\xi_{(k-2)T}})\Big)^2\right)= \left[\ES_x\left(\ES\Big(\Lambda(X^{x})\,|\, {\cal F}_{T}^{W}\Big)\right)^2\right]_{x= \xi_{(k-2)T}}.
\]
Thus, it follows from the definition of $A_{k+1}$ and $B_{k+1}$ that
\begin{align*}
\ES_{k-2}[(\ES_k C_{k+1})^2]-\ES_{k-2}[(\ES_{k-1} C_{k+1})^2]&=\left[\ES_x\left(\ES\left(\Lambda_1(X^x)\,|\, {\cal F}_{2T}^{W}\right)^2-\ES\left(\Lambda_1(X^x)\,|\, {\cal F}_{T}^{W}\right)^2\right)\right]_{x= \xi_{(k-2)T}},
\end{align*}
where
\begin{align*}
\Lambda_1(X^x)& :=\int_0^{2T}F_{_T}(X^x_{u+.})du-\int_T^{2T}f_{_F}(X^x_u)du-\int_0^T\langle \sigma^*\nabla g_{_F}(X^x_u ),  dW_u \rangle\\
&= \int_0^{2T}F_{_T}(X^x_{u+.})du-\int_T^{2T}f_{_F}(X^x_u)du- \left(g_{_F}(X^x_{_T})-g_{_F}(X^x_0)-\int_0^T{\cal A}g_{_F}(X^x_u)du\right).
\end{align*}
Note that the second expression clearly defines a functional on the canonical space. Now,
\begin{align*}
\ES_x\Big(\ES\left(\Lambda_1(X^x)\,|\, {\cal F}_{2T}^{W}\right)^2-\ES\left(\Lambda_1(X^x)\,|\, {\cal F}_{T}^{W}\right)&^2\Big)=\ES_x\left[\left(\ES\left(\Lambda_1(X^x)\,|\, {\cal F}_{2T}^{W}\right)-\ES\left(\Lambda_1(X^x)\,|\, {\cal F}_{T}^{W}\right)\right)^2\right]\\
&=\ES_x\left[\left(\ES\left(\tilde{\Lambda}_1(X^x)\,|\, {\cal F}_{2T}^{W}\right)-\ES\left(\tilde{\Lambda}_1(X^x)\,|\, {\cal F}_{T}^{W}\right)\right)^2\right]
\end{align*}
where $\tilde{\Lambda}_1(X^x)=\Lambda_1(X^x)-\int_0^T f_{_F}(X_u^x)du$. The result follows using that ${\cal F}_s^W={\cal F}_s$, that $(\int_0^{t}\langle \sigma^*\nabla g_{_F}(X^x_u ),  dW_u)_{t\ge0}$ is a $({\cal F}_t)$-martingale and that $\ES[A_{_{2T}}-A_{_T}\,|\,{\cal F}_T]=0$.

\medskip
\noindent $(ii)$ Let $x\in\ER^d$ and set  
$$
\psi(x,.)=\ES\Big(A_{_{2T}}^x\,|\,{\cal F}_{_{2T}}\Big)-\ES\Big(A_{_T}^x\,|\,{\cal F}_{_T}\Big)-\int_T^{2T}\sigma^*\nabla g_F(X_u^x)dW_u.
$$ 
Let $(x_n)$ be a convergent sequence of $\ER^d$ to $x$. Owing to the standard identity $a^2-b^2=(a-b)(a+b)$ and Schwarz's inequality,
$$
|\Psi(x)-\Psi(x_n)|\le\ES[|\psi(x,.)-\psi(x_n,.)|^2]^{\frac{1}{2}} \ES[|\psi(x,.)+\psi(x_n,.)|^2]^{\frac{1}{2}} 
$$
Let ${\cal S}=\{x\}\cup\{x_n,n\ge1\}$. Since $F$ and $f_{_F}$ are bounded,  
\begin{align*}
\ES[|\psi(x,.)+\psi(x_n,.)|^2]&\le C\left(1+ \sup_{v\in{\cal S}}\ES\left[\left(\int_T^{2T}\sigma^*\nabla g_F(X_u^{v})dW_u\right)^2\right]\right)\\
&=C(1+\sup_{v\in{\cal S}}\int_T^{2T}\ES[|\sigma^*\nabla g_{_F}(X_u^{v})|^2]du\\
&\le C(1+\sup_{v\in{\cal S}}\sup_{u\in[T,2T]}\ES[V^a(X_u^{v})]\le C(1+\sup_{n\ge1}V^a(x_n)),
\end{align*}
owing to Lemma~\ref{lemmem1}$(i)$. Thus,
$$
|\Psi(x)-\Psi(x_n)|\le C\ES[|\psi(x,.)-\psi(x_n,.)|^2]^{\frac{1}{2}}
$$
and it follows easily that $\Psi$ will be continuous if both $x\mapsto \ES_x\left(\ES\Big(A_{_{iT}}^x\,|\,{\cal F}_{_{iT}}\Big)\right)^2$ ($i=1,2$) and $x\mapsto \ES_x\left(\int_T^{2T}\sigma^*\nabla g_F(X_u^x)dW_u\right)^2,$ are continuous. On the one hand $F$ and $f_{_F}$ being Lipschitz continuous, elementary computations show that for $i=1,2$,
\begin{eqnarray*}
 \ES \left|\int_0^{iT}  \ES[A_{_{iT}}^x\,|\, {\cal F}_{iT}]du-\int_0^{iT}  \ES[A_{_{iT}}^{x_n}\,|\, {\cal F}_{{iT}}]du\right|^2&\le&\left\|\sup_{t\in[0,3T]}|X^{x}_t-X^{x_n}_t|\right\|_2^2.\\
 \end{eqnarray*}
Now, since $b$ and $\sigma$ are Lipschitz continuous functions, for every $p>0$, there exists a real constant $C_{b,\sigma,p,T}>0$ such that (see $e.g.$ ~\cite{kunita} or ~\cite{PRO}),
\[
\left\|\sup_{t\in[0,3T]}|X^{x}_t-X^{x_n}_t|\right\|_p^p\le C_{b,\sigma,p,T}|x-x_n|^p.
\]
The continuity of $x\!\mapsto\! \ES_x\left(\ES\Big(A_{_{iT}}^x\,|\,{\cal F}_{_{iT}}\Big)\right)^2$, $i\!=\!1,2$,  follows.
On the other hand using~\eqref{25},
\begin{align*}
\ES\Big| \int_T^{2T}\psg \nabla & g_{_F}(X_u^x),
\sigma(X_u^x)dW_u\psd \int_T^{2T}\psg \nabla g_{_F}(X_u^{x_n}),
\sigma(X_u^{x_n})dW_u\psd\Big|^2\\
&= \int_T^{2T}\ES|\sigma^*\nabla g_{_F}\sigma(X_u^{x})-\sigma^*\nabla g_{_F}(X_u^{x_n})|_2^2du\\
&\le  C\left(\ES[1+\sup_{u\in[T,2T]}|X_u^x|^4]^\frac{1}{2}\ES[\sup_{u\in[T,2T]}|X_u^{x}-X_u^{x_n}|^4]^\frac{1}{2}\right)\le C (1+|x|^2)|x-x_n|^2.
\end{align*}
This concludes the proof.
 \end{proof}
\begin{prop}\label{propTCL}
Suppose that assumptions of Theorem~\ref{thprincipal}(a) hold. Then, 
\begin{equation}
\frac{M_n}{\sqrt{n}}\xrightarrow{{\cal L}}{\cal N}(0,\sigma_{_F}^2)\quad \mbox{ as }\quad n\rightarrow+\infty.
\end{equation}
\end{prop}
\begin{proof}
By Lemma~\ref{lemme3}, 
\begin{equation*}
\frac{1}{n}\sum_{k=2}^n\ES_{k-2} (\Delta M_k)^2\xrn{n\nrn}\sigma_{_F}^2\quad a.s.
\end{equation*}
Then, we only need to prove a Lindeberg type condition (see~\cite{hall}, Corollary 3.1). To be precise, we will show that for every $\varepsilon>0$
\begin{equation*}
\frac{1}{n}\sum_{k=1}^n \ES_{k-1}[|\Delta M_k|^2{\bf{1}}_{\{|\Delta M_k|\ge \varepsilon\sqrt{n}\}}]\xrn[n\nrn]{\PE}0.
\end{equation*}
First, a martingale argument similar to that of the beginning of the proof of Lemma~\ref{lemme1}, yields that
\begin{equation*}
\frac{1}{n}\sum_{k=2}^n \left(\ES_{k-1}[|\Delta M_k|^2{\bf{1}}_{\{|\Delta M_k|\ge \varepsilon\sqrt{n}\}}]-
\ES_{k-2}[|\Delta M_k|^2{\bf{1}}_{\{|\Delta M_k|\ge \varepsilon\sqrt{n}\}}]\right)\xrn[n\nrn]{\PE}0.
\end{equation*}
Second, using conditional H\"older and Chebyschev inequalities, we have for every $\varepsilon,\, \delta>0$
$$
\ES_{k-2}[|\Delta M_k|^2{\bf{1}}_{\{|\Delta M_k|\ge \varepsilon\sqrt{n}\}}]\le \frac{1}{(\varepsilon n)^{2\delta}}\ES_{k-2}[|\Delta M_k|^{2(1+\delta)}],
$$
and thanks to~\eqref{bk1} and~\eqref{bk2}, we deduce that 
$$
\ES_{k-2}|\Delta M_k|^{2(1+\delta)}     {\bf{1}}_{\{|\Delta M_k|  \ge \varepsilon\sqrt{n}\}}\le C \bar{G}_{k-2,a(1+\delta)}(\xi_{(k-2)T})\le\frac{C_{\varepsilon}}{ n^{2\delta}} 
V^{a(1+\delta)}(\xi_{(k-2)T}).
$$
Thus, taking $\delta\!\in(0, \frac{\frac p2 -1}{a})$ so that $a(1+\delta)\le p/2+a-1$, we have for every $\delta>0$, $a.s.$,
\begin{equation*}
\limsup_{n\rightarrow+\infty}\frac{1}{n}\sum_{k=1}^n \ES_{k-1}[|\Delta M_k|^2{\bf{1}}_{\{|\Delta M_k|\ge \varepsilon\sqrt{n}}]
\le C_\varepsilon \limsup_{n\rightarrow+\infty}\frac{1}{n^{1+2\delta}}\sum_{k=1}^n V^{a(1+\delta)}(\xi_{(k-2)T})=0,
\end{equation*} 
by applying Lemma~\ref{lemme0}$(ii)$.
 \end{proof}
\section{Study of $(\Theta_{n,1})$, $(\Theta_{n,2})$ and $(\Theta_{n,3})$}\label{Thetastudy}
In this section, we focus on the remainder terms of the decomposition of the error (see Lemma~\ref{lemme1}). Owing to Proposition~\ref{propTCL},
it is now enough to prove that 
$$
\frac{\Theta_{n,i}}{\sqrt{n}}\xrn{\PE} 0\quad  \mbox{ as } \quad n\nrn 
\quad  \textnormal{ for }i=1,2,3.
$$
where $\xrn{\PE} $ denotes the convergence in probability. For $i=1, 2$, these properties are stated in Lemma~\ref{lemme4} and~\ref{lemme5}. For $i=3$, the result is obvious.

\begin{lemme}\label{lemme4}
Assume $b$ and $\sigma$  are Lipschitz continuous functions 
such that  $\mathbf{(S_{a,p})}$ with parameters $a\in(0,1]$, $p>2$, and an essentially quadratic Lyapunov function  $V$  satisfying  $\liminf_{|x|\rightarrow+\infty} V^{p+a-1}(x)/|x|>0$. Let $F:\mathbb{D}_{uc}(\ER_+,\ER^d)\rightarrow\ER$ be Lipschitz continuous.  If  the step condition~(\ref{condpas33}) holds 
then
$$
\frac{\Theta_{n,1}}{\sqrt{n}}\xrn{\PE} 0\quad  \mbox{ as } \quad n\nrn. 
$$
\end{lemme}
\noindent {\it Proof.} 
Since $F$ is Lipschitz continuous, it follows from Lemma~\ref{lemmem1}$(ii)$ (applied with $p=1$) that, for every $u\!\in I_k$, 
\begin{eqnarray*}
\ES_{k-1}\left|F_{_T}({\xi}^{(\usim{u})})-F_{_T}\big(X^{(\usim{u}),\xi_{\usim{u}}}\big)\right|&\le& [F]_{{\rm Lip}}\ES_{k-1} \left[\sup_{t\in[0,T]}|{\xi}^{(\usim{u})}_t-X^{(\usim{u}),\xi_{\usim{u}}}_t|\right]\\
&\le& C_{b,\sigma,T,F}\sqrt{\gamma_{N(\usim{u})}}(1+\ES_{k-1}|\xi_{\usim{u}}|).
\end{eqnarray*}
Consequently, 
\begin{align*}
|\Theta_{n,1}| &\le  C_{b,\sigma,T,F}\sum_{k=1}^n  \int_{I_k}\sqrt{\gamma_{N(\usim{u})+1}}du\,(1+\ES_{k-1}\sup_{v\in I_k}|\xi_{v}|)\\&\le C_{b,\sigma,T,F}\sum_{k=1}^n 
\int_{I_k}\sqrt{\gamma_{N(\usim{u})+1}}du\,(1+|\xi_{(k-1)T}|)
\end{align*}
where in the second inequality, we used Lemma~\ref{lemmem1}$(i)$. Since $\liminf_{|x|\to +\infty} V^{p+a-1}(x)/|x|>0$   and $N(\usim{u})=N(u)$, we deduce that
\begin{equation}\label{CTheta1}
\frac{|\Theta_{n,1}|}{\sqrt{n}}\le \frac{C}{\sqrt{n}}\sum_{k=1}^n  \int_{I_k}\sqrt{\gamma_{N(u)+1}}du\,(1+V^{p+a-1}(\xi_{(k-1)T})).
\end{equation}
Thus, owing to the Kronecker Lemma, 
$$\frac{\Theta_{n,1}}{\sqrt{n}}\xrn{n\nrn}0\quad a.s.\quad\textnormal{if}\quad\sum_{k=1}^n  \delta_k\,(1+V^{p+a-1}(\xi_{(k-1)T}))<+\infty\quad a.s.,$$
with 
$$
\delta_k=\frac{1}{\sqrt{k}}\left(\int_{I_k}\sqrt{\gamma_{N(u)+1}}du\right).
$$
Now, as  $(\delta_k)$ is nonincreasing, it follows from Lemma~\ref{lemme0}$(i)$ that it is now enough to show that 
$\sum_{k\ge1}\delta_k<+\infty$. We have
\begin{equation*}
\sum_{k\ge1}\delta_k \le C\left(1+\int_{\gamma_1}^{+\infty}\sqrt{\frac{\gamma_{N(u)+1}}{u}}du\right)
\end{equation*}
and 
\[
\int_{\gamma_1}^{+\infty}\sqrt{\frac{\gamma_{N(u)+1}}{u}}du\le\sum_{\ell\ge1}\sqrt{\gamma_{\ell+1}}\int_{\Gamma_\ell}^{\Gamma_{\ell+1}}\frac{1}{\sqrt{u}}du\le \sum_{\ell\ge1}\sqrt{\gamma_{\ell+1}}\frac{\gamma_{\ell+1}}{\sqrt{\Gamma_\ell}}.
\]
Using that the step sequence $(\gamma_n)$ is nonincreasing, we deduce from Condition~\eqref{condpas33} that
\begin{equation}\label{contserie2}
\int_{\gamma_1}^{+\infty}\sqrt{\frac{\gamma_{N(u)+1}}{u}}du\le \sum_{\ell\ge1}\frac{\gamma_{\ell}^{\frac{3}{2}}}{\sqrt{\Gamma_\ell}}<+\infty. \hskip 2 cm  \Box
\end{equation}
\begin{lemme}\label{lemme5}
Assume $b$ and $\sigma$  are Lipschitz continuous functions 
satisfying  $\mathbf{(S_{a,p})}$ with an essentially quadratic Lyapunov function  $V$ and parameters $a\in(0,1]$ and $p>2$.  Let $F:\mathbb{D}_{uc}([0,T],\ER^d)\to \ER$ be a functional satisfying $\mathbf{(C_F^1)}$ and $\mathbf{(C_F^2)}$.
If the step condition~(\ref{condpas33}) holds, then
$$
\frac{\Theta_{n,2}}{\sqrt{n}}\xrn{\PE} 0\quad  \mbox{ as } \quad n\nrn.
$$
\end{lemme}
\begin{proof} Owing to the It\^o formula, we have:
\begin{align*}
&g_{_F}(\xi_{kT})-g_{_F}(\xi_{(k-1)T})=\int_{I_k} \bar{{\cal A}}g_{_F}(\xi_u,\xi_{\underline{u}}) du+\int_{I_k}\psg \nabla g_{_F}(\xi_u),\sigma(\xi_{\underline{u}})dW_u\psd\quad\textnormal{where},\\
&\bar{{\cal A}}g_{_F}(x,y)=\psg\nabla g_{_F}(x),b(y)\psd+\frac{1}{2}{\rm Tr}\left(\sigma^*(y)D^2 g_{_F}(x)\sigma(y)\right).
\end{align*}
Then, it follows from the definition of $\Theta_{n,2}$ that
\begin{align*}
\Theta_{n,2}&=\sum_{k=1}^n g_{_F}(\xi_{kT})-g_{_F}(\xi_{(k-1)T})+\int_0^{nT}
\left({\cal A} g_{_F}(\xi_{\usim{u}})-\bar{{\cal A}}g_{_F}(\xi_u,\xi_{\underline{u}})\right)du\\
&+\int_{0}^{nT} \psg \sigma^*(\xi_{\underline{u}})\nabla g_{_F}(\xi_u)-\sigma^*(\xi_{\usim{u}})\nabla g_{_F}(\xi_{\usim{u}}),dW_u\psd.
\end{align*}
Since $g_{_F}$ is bounded,
$$
\frac{1}{\sqrt{n}}\sum_{k=1}^n g_{_F}(\xi_{kT})-g_{_F}(\xi_{(k-1)T})=\frac{g_{_F}(\xi_{nT})-g_{_F}(\xi_{0})}{\sqrt{n}}\xrn{a.s.} 0 \quad \mbox{ as } \quad n\nrn.
$$
Then, it is now enough to show that 
\begin{equation}\label{amoinsabar}
\frac{1}{\sqrt{n}}\int_0^{nT}
\left({\cal A} g_{_F}(\xi_{\usim{u}})-\bar{{\cal A}}g_{_F}(\xi_u,\xi_{\underline{u}})\right)du\xrn{L^1} 0 \quad \mbox{ as } \quad n\nrn
\end{equation}
and that,
\begin{equation}
\label{martrij}
\frac{1}{\sqrt{n}} \int_{0}^{nT} \psg \sigma^*(\xi_{\underline{u}})\nabla g_{_F}(\xi_u)-\sigma^*(\xi_{\usim{u}})\nabla g_{_F}(\xi_{\usim{u}}),dW_u\psd\xrn{a.s.} 0 \quad \mbox{ as } \quad n\nrn.
\end{equation}
First, using that $g_{_F}$ is a bounded ${\cal C}^2$-function with bounded Lipschitz continuous derivatives, that $b$ and $\sigma$ are Lipschitz continuous functions, one checks that
\begin{align*}
\left|{\cal A}g_{_F}(\usim{x})-\bar{{\cal A}}g_{_F}(x,\underline{x})\right|\le& \;C\left(|\usim{x}-x|.|b(\underline{x})|+|\usim{x}-\underline{x}|+
|\usim{x}-\underline{x}|^2+\|\sigma(\underline{x})\|^2.|\usim{x}-x|\right).
\end{align*}
Then, using that
$$\max\left(|\xi_{\usim{u}}-\xi_{\underline{u}}|,|\xi_u-\xi_{\usim{u}}|\right)\le 2\sup_{v\in[\Gamma_{N(u)},\Gamma_{N(u)+1})}|\xi_v-\xi_{\underline{u}}|,$$
it follows from Lemma~\ref{lemmem1}$(iii)$ applied with $\phi=V^a$, $p=1$ and $p=2$, that,
\begin{align*}
\ES\Big[\big|{\cal A} g_{_F}(\xi_{\usim{u}})-\bar{{\cal A}}g_{_F}&(\xi_u,\xi_{\underline{u}})\big||{\cal F}_{\underline{u}}\Big]\\
&\le C\left(\sqrt{\gamma_{N(u)+1}}\,V^{\frac{a}{2}}(\xi_{\underline{u}})(1+|b(\xi_{\underline{u}})|+\|\sigma(\xi_{\underline{u}})\|^2)+\gamma_{N(u)+1}V^{a}(\xi_{\underline{u}})\right).
\end{align*}
By Asssumption $\mathbf{(S_{a,p})}$, we deduce that 
$$
\frac{1}{\sqrt{n}}\ES\left[ \int_0^{nT}
\left|{\cal A} g_{_F}(\xi_{\usim{u}})-\bar{{\cal A}}g_{_F}(\xi_u,\xi_{\underline{u}})\right|du\right]
\le \frac{C}{\sqrt{n}}\int_0^{nT} \ES[V^{\frac{3}{2}a}(\xi_{\underline{u}})]\sqrt{\gamma_{N(u)+1}}du.
$$
Now, since $p\ge 2$, $\frac{3}{2}a\le p+a-1$, and by 
~\eqref{contserie2}, we have
\begin{equation*}
\int_1^\infty\sqrt{\frac{\gamma_{N(u)+1}}{u}}du\le C\sum_{k\ge 1}\frac{\gamma_{k}^{\frac{3}{2}}}{\sqrt{\Gamma_k}}<+\infty.
\end{equation*}
Then~\eqref{amoinsabar}  follows from Lemma~\ref{lemme0}$(i)$ and the Kronecker Lemma like in the proof of Lemma~\ref{lemme4}.

\noindent Second, we focus on~\eqref{martrij}. Set $Z_0=0$ and 
$$
Z_n=\sum_{k=1}^n\frac{1}{\sqrt{k}}\int_{I_k} \psg \sigma^*(\xi_{\underline{u}})\nabla g_{_F}(\xi_u)-\sigma^*(\xi_{\usim{u}})\nabla g_{_F}(\xi_{\usim{u}}),dW_u\psd,\quad\  n\ge1.
$$
The sequence $(Z_n)$ being a $({\cal G}_n)$-adapted martingale, it follows from Doob's convergence Theorem for $L^2$-bounded martingales  
that~\eqref{martrij} holds if 
\begin{equation}\label{convmartrij}
\sup_{n\ge1}\ES[(Z_n)^2]<+\infty.
\end{equation}
Let us show~\eqref{convmartrij}. First,
\begin{align*}
\ES[(Z_n)^2]&= \sum_{k\ge1}\frac{1}{k}\int_{I_k}\ES\left[\left|\sigma^*(\xi_{\un{u}})\nabla g_{_F}(\xi_{u})-\sigma^*\nabla g_{_F}(\xi_{\usim{u}})\right|^2\right]du.
\end{align*}

\noindent By similar arguments as for~\eqref{25},
$$
\left|\sigma^*({\un{x}})\nabla g_{_F}(x)-\sigma^*\nabla g_{_F}({\usim{x}})\right|^2\le C\left(1+\|\sigma^*(\un{x})\|^2\right)\left(|x-\usim{x}|^2+|\un{x}-\usim{x}|^2\right).
$$
Then, owing to the  fact  that $\usim{u}\!\in [\un{u},u]$,   it follows from $\mathbf{(S_{a,p})}$ and Lemma~\ref{lemmem1}$(iii)$ that
$$
\ES\left[\left|\sigma^*(\xi_{\un{u}})\nabla g_{_F}(\xi_{u})-\sigma^*\nabla g_{_F}(\xi_{\usim{u}})\right|^2\right]
\le C\gamma_{N(u)+1}\ES[V^{2a}(\xi_{\un{u}})].
$$
Thus, since $u\le kT$ for every $u\in I_k$,
\begin{align*}
\sum_{k\ge1}\ES[|\Delta Z_k|^2]&\le \ES[|Z_1|^2]+C\sum_{k\ge2}\int_{I_k}\ES[V^{2a}(\xi_{\un{u}})]\frac{\gamma_{N(u)+1}}{u}du\\
&\le C\left(1+\int_1^{+\infty}\ES[V^{2a}(\xi_{\un{u}})]\frac{\gamma_{N(u)+1}}{u}du\right).
\end{align*}
Finally, by  a similar argument to~\eqref{contserie2}, we
have 
$$\int_1^{+\infty}\frac{\gamma_{N(u)+1}}{u}du\le C\sum_{k\ge1}\frac{\gamma_k^2}{\Gamma_k}<+\infty$$
and~\eqref{convmartrij} follows from Lemma~\ref{lemme0}$(i)$ and from the fact  that $ 2a\le p+a-1$ when $p\ge2$.
\end{proof}
\section{Proof of the main theorems}\label{PMT}
The first step for the proof of these theorems is now to state our main result 
about the sequence $({\cal P}^{(n,T)}(\omega,F_{_T}))_{n\ge1}$ studied in the two previous sections:
\begin{prop}\label{proppnt}
Let $T>0$, $a\in(0,1]$ and $p>2$. Assume $b$ and $\sigma$ are Lipschitz continuous functions  satisfying $\mathbf{(S_{a,p})}$ with an essentially quadratic Lyapunov function $V$ such that $\liminf_{|x|\rightarrow+\infty}V^{p+a-1}(x)/|x|>0$. Assume $\mathbf{(S^\nu_T)}$. 
Let $F:\mathbb{D}_{uc}([0,T],\ER^d)$ be a functional satisfying $\mathbf{(C_F^1)}$ and $\mathbf{(C_F^2)}$. Finally, assume that the step sequence $(\gamma_n)_{n\ge 1}$ satisfies~(\ref{nonincreagam}) and~(\ref{condpas33}).
Then, 
\begin{equation}\label{pntresult}
{\sqrt{nT}}\left(\pnt(\omega,F_{_T})-\PE_\nu(F_{_T})\right)\xrn[n\rightarrow+\infty]{\cal L}{\cal N}\left(0,\sigma_{_F}^2\right).
\end{equation}
\end{prop}
\begin{proof} Owing respectively to Lemma~\ref{lemme4}, Lemma~\ref{lemme5} and the fact that $F$ is bounded, $\Theta_{n,1}$, $\Theta_{n,2}$ and $\Theta_{n+1,3}$ defined in Lemma~\ref{lemme1} satisfy:
$$
\frac{\Theta_{n,1}+\Theta_{n,2}+\Theta_{n+1,3}}{\sqrt{nT}}\xrn{\PE}0 \quad \mbox{ as } \quad n\nrn.
$$
Then, the proposition follows from  Proposition~\ref{propTCL} and from the decomposition of  $\pnt-\PE_\nu(F_{_T})$ stated in Lemma~\ref{lemme1}.
\end{proof}
\noindent We are now able to prove Theorems~\ref{thprincipal} and~\ref{thprincipal3}.

\medskip
\noindent \textbf{Proof of Theorems~\ref{thprincipal}$(a)$ and~\ref{thprincipal3}.}
First, let $(t_k)_{k\ge1}$ denote a sequence of positive real numbers such that $t_k\rightarrow+\infty$. Set $\displaystyle{n_k=\lfloor \frac{t_k}{T}\rfloor}$. Since $F_{_T}$ is a bounded functional, we have:
\begin{align*}
\Big|\sqrt{t_k}\Big(\frac{1}{t_k}\int_0^{t_k} F_{_T}(\xi^{(\usim{u})})du&-\PE_\nu(F_{_T})\Big)-\sqrt{n_k T} \left(\pnt(\omega,F_{_T})-\PE_\nu(F_{_T})\right)\Big|\\
&\le 2\|F_{_T}\|_\infty(\sqrt{t_k}-\sqrt{n_k}T)+\|F_{_T}\|_\infty\frac{t_k-n_kT}{\sqrt{n_kT}}\xrn{k\nrn}0\quad a.s.
\end{align*}
Thus, Theorem~\ref{thprincipal3} follows taking $\usim{u}=u$. For Theorem~\ref{thprincipal}(a), setting $\usim{u}=\un{u}\vee (\lfloor u/T\rfloor T)$, and $t_n=\Gamma_n$, we obtain that
$$
\sqrt{\Gamma_n}\left(\frac{1}{\Gamma_n}\int_0^{\Gamma_n} F_{_T}(\xi^{(\un{u}\vee (\lfloor\frac{u}{T}\rfloor T))})du-\PE_\nu(F_{_T})\right)\xrn{\cal L}{\cal N}\left(0,\sigma_{_F}^2\right)0 \quad \mbox{ as } \quad n\nrn.
$$
Now, 
$$
\sqrt{\Gamma_n}\left|\bar{\nu}^{(n)}(\xi(\omega),F_{_T})-\frac{1}{\Gamma_n}\int_0^{\Gamma_n} F_{_T}(\xi^{(\un{u}\vee (\lfloor\frac{u}{T}\rfloor T))})du\right|
\le  \frac{\|F_{_T}\|_\infty}{\sqrt{\Gamma_n}}\sum_{k=1}^{\lfloor \Gamma_n \rfloor}\gamma_{N(k)+1},
$$
and the fact that $\displaystyle{\gamma_{N(k)+1}\le C\sum_{i=N(k-1)+1}^{N(k)}\gamma_i^2}$ implies that,
$$
\frac{1}{\sqrt{\Gamma_n}}\sum_{k=1}^{\lfloor \Gamma_n \rfloor}{\gamma_{N(k)+1}}\le  \frac{1}{\sqrt{\Gamma_n}}\sum_{i=1}^{n}\gamma_i^2.$$
By~\eqref{condpas33} and the Kronecker Lemma,
$$
\frac{1}{\sqrt{\Gamma_n}}\sum_{i=1}^{n}\gamma_i^s\xrn{n\nrn}0\quad\forall s\ge3/2.
$$
Applying this identity with $s=2$ yields the result.\hfill $\Box$

\medskip
\noindent\textbf{Proof of Theorem~\ref{thprincipal}$(b)$.} Owing to Theorem~\ref{thprincipal}$(a)$, it is now enough to show that
$$
\sqrt{\Gamma_n}\left(\bar{\nu}^{(n)}(\bar{X}(\omega),F_{_T})-\bar{\nu}^{(n)}(\xi(\omega),F_{_T})\right)\xrn[\PE]{n\nrn}0.
$$
Since $F_{_T}$ is a Lipschitz bounded functional, it follows from the definition of the previous occupation measures that
$$
\sqrt{\Gamma_n}\ES\left[\left|\bar{\nu}^{(n)}(\bar{X}(\omega),F_{_T})-\bar{\nu}^{(n)}(\xi(\omega),F_{_T})\right|\right]\le\frac{[F_{_T}]_{{\rm Lip}}}{\sqrt{\Gamma_n}}
\int_0^{\Gamma_n}\ES[\sup_{s\in[0,T]}|\xi_{\underline{u}+s}-\xi_{\un{\un{u}+s}}|]du.
$$
By Lemma~\ref{lemmem1}$(iv)$ and Jensen's inequality, for every $q>1$,
\begin{align*}
\ES[\sup_{s\in[0,T]}|\xi_{\un{u}+s}-\xi_{\un{\un{u}+s}}||{\cal F}_{\un{u}}]&\le \ES[\sup_{s\in[0,T]}|\xi_{\un{u}+s}-\xi_{\un{\un{u}+s}}|^q|{\cal F}_{\un{u}}]^\frac{1}{q}\\&
\le C\left(V^\frac{aq}{2}(\xi_{\un{u}})\gamma_{N(u)+1}^{\frac{q}{2}-1}\right)^{\frac{1}{q}}\le CV^{\frac{a}{2}}(\xi_{\un{u}})\gamma_{N(u)+1}^{\frac{1}{2}-\frac{1}{q}}.
\end{align*}
Thus, we deduce that
$$
\int_0^{\Gamma_n}\ES[\sup_{s\in[0,T]}|\xi_{u+s}-\xi_{\un{\un{u}+s}}|]du\le C\sum_{k=1}^{n-1}\gamma_{k}^{\frac{3}{2}-\frac{1}{q}}\ES[V^{\frac{a}{2}}(\xi_{_{\Gamma_{k-1}}})].
$$
Let $\delta$ be a positive number such that~\eqref{condpas34} holds. Taking $q$ such that $1/q\le\delta$,
we deduce from~\eqref{condpas34} and Lemma~\ref{lemme0}$(i)$ that
$$
\sum_{k\ge1}\frac{\gamma_{k}^{\frac{3}{2}-\frac{1}{q}}}{\sqrt{\Gamma_k}}\ES[V^{\frac{a}{2}}(\xi_{_{\Gamma_{k-1}}})]=
\int_{0}^{+\infty}\ES[V^{\frac{a}{2}}(\xi_{\un{u}})] \frac{\gamma_{N(u)+1}^{\frac{3}{2}-\frac{1}{q}}}{\sqrt{\Gamma_{N(u)+1}}}du<+\infty.
$$
We again deduce the result from Kronecker's Lemma.\hfill $\Box$

\medskip
\noindent\textbf{Proof of Theorem~\ref{thprincipal2}.} We only give the main ideas of the proof of 
this result about the ``perfect Euler scheme'' $(X_t)$, that is naturally simpler than that of the discretized processes. First, the reader can check that setting 
$$\tilde{\cal P}^{(n,T)}(\omega,F_{_T})=\frac{1}{nT}\int_0^{nT} F_{_T}(X^{(u)})du,$$
one obtains a similar decomposition as that of Lemma~\ref{lemme1} replacing $\usim{u}$ by $u$ and $\phi_{_F}$ by $\tilde{\phi}_F$ defined by
\begin{equation*}
\tilde{\phi}_{_F}(1)= 0 \quad \textnormal{ and } \quad 
\int_{I_{k-1}} F_{_T}(X^{(u)})du  \quad \textnormal{ if $k\ge 2$.}
\end{equation*}
The main difference in this decomposition is that the term corresponding to $\Theta_{n,1}$ is null. Then, since the assumption
$\liminf_{|x|\rightarrow+\infty} V^{p+a-1}(x)/|x|>0$ is only needed  in the proof of the result about $\Theta_{n,1}$ (see Lemma \ref{lemme4}), we deduce that  it is not necessary here. Then, the sequel of the proof works since the statements of Lemma~\ref{lemme0} still hold if one replaces
 $\xi$ by $X$. To be  precise, the first statements of $(i)$ and $(ii)$ can be directly derived from~\cite{panloupthese} (Chapter 1) and the second ones from an adaptation of the proof of this lemma.  \hfill $\Box$

\section{Numerical Test on Barrier Options in the Heston model}\label{simulations}
As shown in~\cite{PP1}, our algorithm can be successfully implemented   for pricing path-dependent options in stochastic volatility models when the volatility process evolves in its stationary regime. Furthermore, such stationary versions of stochastic volatility models are more performing to take into account the behaviour of implicit volatility  for short maturities. Then, even if the assumptions of our main theorems are usually not satisfied for the functionals involved in this context, we choose in this section to
illustrate them by such an example. To be precise, we test numerically the asymptotic normality obtained in the main results on the computation of several Barrier options in a Heston stationary stochastic volatility model. The dynamics of the traded asset price process $(S_t)_{t\ge0}$ is given by:  
\begin{align*}
&dS_t=S_t(rdt+\sqrt{(1-\rho^2)v_t}dW^1_t+\rho\sqrt{v_t} dW_t^2),\quad S_0=s_0>0,\\
&dv_t=k({\theta}-v_t)dt+\varsigma\sqrt{v_t}dW_t^2,\quad v_0>0,
\end{align*}
where $r$ denotes the interest rate, $(W^1,W^2)$ is a standard
two-dimensional Brownian motion, $\rho\in[-1,1]$ and $k$, $\theta$
and $\varsigma$ are some nonnegative numbers. This model was
introduced by Heston (\cite{heston}). The equation for
$(v_t)$ has a unique (strong) pathwise continuous solution living
in $\ER_+$. If moreover, $2k\theta>\varsigma^2$ then, $(v_t)$ is a
positive process (see~\cite{lamblapeyre}). In this case, the volatility process $(v_t)$
has a unique invariant probability  ${\nu_0}$ with {\em gamma} distribution, namely
${\nu_0}=\gamma(a,b)$ with $a=(2k)/\varsigma^2$ and
$b=(2k\theta)/\varsigma^2$. Thus, we assume that
$(v_t)$ evolves in its stationary regime, $i.e.$ that $${\cal
L}(v_0)={\nu_0}.$$
Under this assumption, we showed in~\cite{PP1} that any option premium can be expressed  as the
expectation of a functional of a two-dimensional stationary
stochastic process.  Let us recall the idea: we will write $(S_t)$ as a functional of a stationary process. Elementary It\^o calculus yields
\begin{equation}\label{hestoneq}
S_t=s_0\exp\Big(rt-\frac{1}{2}\int_0^tv_sds+\rho\int_0^t\sqrt{v_s}dW_s^2+\sqrt{1-\rho^2}\int_0^t\sqrt{v_s}dW^1_s\Big).
\end{equation}
Introducing the  $2$-dimensional $SDE$,
\begin{equation}
\begin{cases}
dy_t=-y_tdt+\sqrt{v_t}dW_t^1,\\
dv_t=k({\theta}-v_t)dt+\varsigma\sqrt{v_t}dW_t^2,
\end{cases}
\end{equation}
and using the fact that
$$
 \int_0^t\sqrt{v_s}dW^1_s=y_t-y_0+\int_0^t y_sds\quad\textnormal{and}\quad \int_0^t\sqrt{v_s}dW_s^2=\frac{v_t-v_0-k\theta
t+k\int_0^t v_sds}{\varsigma},
$$
we deduce that we can construct a (continuous) map $\Phi$ from ${\cal C}(\ER_+,\ER^2)$ to ${\cal C}(\ER_+,\ER)$ such that
$(S_t)_{t\ge0}=\Phi((y_t,v_t)_{t\ge0})$. Now, we have built $(y_t)$ so that $(y_t,v_t)_{t\ge0}$ has a stationary regime. Denoting by $\mu$ the invariant distribution of $(y_t,v_t)$, we obtain that
$$
\ES[F(S_t,0\le t\le T)]=\ES_\mu[F\circ\Phi((y_t,v_t),0\le t\le T)].
$$
For further details we refer to~\cite{PP1}. Here, we are interested with an {\em Up-and-Out}  barrier option whose discounted payoff is given by:
$$
F(S_t,0\le t\le T)=e^{-rT}\left(S_T-K\right)_+ {\bf 1}_{\{\sup_{0\le t\le T} S_t\le L\}}
$$
where $L>K>0$. We now specify the discretization. First,  the genuine Euler scheme of the so-called Heston volatility process (also known as the Cox-Ingersoll-Ross process) $(v_t)$ cannot be implemented since it does note preserve the positivity. Thus, we must replace it by a specific discretization scheme: we denote by $(\bar{v}_t)$ the stepwise constant Euler scheme built as follows:  
$$
\bar{v}_{_{\Gamma_{n+1}}}=\big|\bar{v}_{_{\Gamma_n}}+k\gamma_{n+1}
(\theta-\bar{v}_{_{\Gamma_n}})+\varsigma\sqrt{\bar{v}_{_{\Gamma_n}}}(W^2_{_{\Gamma_{n+1}}}-W^2_{_{\Gamma_n}})\big|\quad \mbox{ and }\quad \bar{v}_0=x>0.
$$
Note that convergence  properties of this scheme have been studied in a constant step framework in \cite{berkaoui} (see also \cite{delstra},  \cite{alfonsi}, and \cite{andersen} for other specific discretization schemes).

\smallskip
\noindent Second, we denote by $(\xi_t)$ the continuous discretization scheme of $(\log(\frac{S_t}{s_0}))_{t\ge 0}$ defined by $\xi_0=0$ and
\begin{equation}\label{contructxi}
\xi_t=\xi_{_{\Gamma_n}}+(r-\frac{1}{2}\bar{v}_{_{\Gamma_n}})t+\rho\sqrt{\bar{v}_{_{\Gamma_n}}}(W_t^2-W^2_{_{\Gamma_n}})+\sqrt{(1-\rho^2)\bar{v}_{_{\Gamma_n}}}(W_t^1-W^1_{_{\Gamma_n}}), \; t\!\in [\Gamma_n, \Gamma_{n+1}], \; n\ge 0.
\end{equation}
Note that we do not need to introduce the Euler of $(y_t)$ since its use is nothing but a theoretical way to justify why an algorithm for the approximation of the stationary regime can be adapted to this context. Finally, in order to compute the supremum of $(\xi_t)$, let us recall the principle of the so-called Brownian Bridge method (transposed to this framework). Set
$$
W^{(\Gamma_n)}_t=\rho (W^1_{_{\Gamma_n+t}}-W^1_{_{\Gamma_n}})+\sqrt{1-\rho^2}(W^2_{_{\Gamma_n+t}}-W^2_{_{\Gamma_n}})
$$
and let $(Y_t^{W,\gamma})$ denote the Brownian Bridge on $[0,\gamma]$ defined by $Y_t^{W,\gamma}= W_t-\frac {t}{\gamma} W_{\gamma}$, $t\!\in [0, \gamma]$. For every $t\in[\Gamma_n,\Gamma_{n+1}]$, we have
$$
\xi_t=\xi_{_{\Gamma_n}}+\frac{\xi_{_{\Gamma_{n+1}}}-\xi_{_{\Gamma_n}}}{\Gamma_{n+1}-\Gamma_n}(t-\Gamma_n)+\sqrt{\bar{v}_{_{\Gamma_n}}} Y_{t}^{W^{(\Gamma_n)},\gamma_{n+1}}.
$$
Using the independence and the Gaussian properties of the Brownian motion, one deduces that, for every $n\ge1$, the processes $(\xi_t)_{t\in[\Gamma_l,\Gamma_{l+1}]}$, $l\in\{0,\ldots,n-1\}$ are conditionally independent given the $\sigma$-field $\sigma((\xi_{\gamma_{_l}},\bar{v}_{_{\Gamma_l}},0\le\l\le n)$  and that
\begin{align*}
{\cal L}\Big((\xi_{t})_{t\in[\Gamma_l,\Gamma_{l+1}]}|(\xi_{_{\Gamma_l}},\xi_{_{\Gamma_{l+1}}},&\bar{v}_{_{\Gamma_l}})=(x_l,x_{l+1},{v}_l)\Big)\\&=
{\cal L}\left(x_l+\frac{x_{l+1}-x_l}{\Gamma_{l+1}-\Gamma_l}t+\sqrt{v_l} Y_{t}^{W,\gamma_{_{l+1}}},t\in[0,\gamma_{_{l+1}}]\right)
\end{align*}
where $W$ denotes a standard Brownian motion.
Then, using the symmetry principle, one can show that, 
for every $x,y\in\ER$, for every $z\ge \max(x,y)$  and positive $\lambda$ and $\gamma,$
$$
\PE(\sup_{t\in[0,\gamma]} x+(y-x)\frac{t}{\gamma}+\lambda Y^{W,\gamma}_t\le z)=1-\exp(-\frac{2}{\gamma \lambda^2}(z-x)(z-y)).
$$
It follows that given $(\xi_{_{\Gamma_l}},\xi_{_{\Gamma_{l+1}}},\bar{v}_{_{\Gamma_l}})$, $\sup_{t\in[\Gamma_l,\Gamma_{l+1}]}\xi_{t}$ can be simulated by the method of inversion of the distribution function. 

\medskip
\noindent Let us now detail the algorithm.

\smallskip
\noindent ${\textbf{{\sc Step} 1}}:$ From $n=0$ to $n=N(T)$. At each step between $n=0$ and $n=N(T)-1$, simulate recursively,  $\bar{v}_{_{\Gamma_{n+1}}}$ and $\xi_{_{\Gamma_{n+1}}}$. Then,  use the Brownian Bridge method to simulate $V_n=\sup_{t\in[\Gamma_n,\Gamma_{n+1}]}\xi_{t}$ given $(\xi_{_{\Gamma_n}},\xi_{_{\Gamma_{n+1}}},\bar{v}_{_{\Gamma_n}})$. Compute recursively $M_n:=\max(V_1,\ldots,V_n)=\max(M_{n-1},V_n)$. At time $N(T)$, compute 
$$
\bar{\nu}^{(1)}(\xi(\omega),F)=e^{-rT}(s_0\exp({\xi}_{_T})-K)_+{\bf 1}_{\{s_0\sup_{t\in[0, T]} exp(\xi_t)\le L\}}.
$$
$$
\vdots\qquad\vdots
$$

\noindent${\textbf{{\sc  Step} i}}:$ From $n=N(T+\Gamma_{i-1})+1$ to $n=N(T+\Gamma_{i}).$ If $M_{N(T+\Gamma_{i-1}+1)}=V_{i-1}$, replace
$M_{N(T+\Gamma_{i-1}+1)}$ by $\max(V_i,\ldots, V_{N(T+\Gamma_{i-1}+1)})$. Store $(\xi_{_{\Gamma_{i-1}}},\ldots,\xi_{_{\Gamma_{N(T+\Gamma_{i-1})+1}}})$ and $(V_i,\ldots, V_{N(T+\Gamma_{i-1}+1)})$. As in Step 1, from $n=N(T+\Gamma_{i-1})+1$ to $n=N(T+\Gamma_{i})$, compute recursively $\bar{v}_{_{\Gamma_{n+1}}}$, $\xi_{_{\Gamma_{n+1}}}$, $V_n$ and the maximum of $V_i, V_{i+1},\ldots, V_n$. Then, at time $N(T+\Gamma_{i})$,
\begin{align*}
\bar{\nu}^{(i)}&(\xi(\omega),F)=\bar{\nu}^{(i-1)}(\xi(\omega),F)\\&+\frac{\gamma_{i+1}}{\Gamma_i}\left(e^{-rT}(s_0\exp({\xi_{_T}}-\xi_{_{\Gamma_{i-1}}})-K)_+{\bf 1}_{\{\underset{t\in[\Gamma_{i-1}, \Gamma_{N(T+\Gamma_{i-1})+1}]}{\sup} \hskip -1 cm  s_0\exp(\xi_t-\xi_{_{\Gamma_{i-1}}})\le L\}}-\bar{\nu}^{(i-1)}(\xi(\omega),F)\right).
\end{align*}

\noindent For the following choices of parameters,
\begin{equation}\label{paramfunc}
s_0=50,\quad r=0.05 ,\quad  T=1 ,\quad\rho=0.5
,\quad\theta=0.01,\quad \varsigma=0.1,\quad k=2,\quad K=50,\quad L=55, 
\end{equation}
we want now to obtain an approximation of the distribution of  the  (asymptotically normal) normalized error
$$
{\cal E}_N:=\sqrt{\Gamma_N}\left(\bar{\nu}_N(\xi(\omega),F)-e^{-rT}\ES[\left(S_T-K\right)_+ {\bf 1}_{\{\sup_{0\le t\le T} S_t\le L\}}]\right)
$$
First, we need to have an accurate approximation of the (risk-neutral) price. In this way, we choose to combine a very long simulation with a variance reduction method taking the corresponding Barrier option in the Black-Scholes model as a control variable. Indeed, on the one hand, it is well-known that the price of such Barrier option has a closed form in the Black-Scholes model (based on the Black-Scholes formula for European options) and on the other hand, this price can be approximated using the algorithm described above by simply replacing the stochastic volatility $(\bar{v}_{t})$ by a constant volatility denoted by $\sigma$. Note that the natural choice for $\sigma$ is the long term volatility $\theta$ which is   the mean of the stationary volatility process $(\bar{v}_{t})$ as well. Then, denoting by $(\xi_t^{{BS}})$ the genuine Euler discretization scheme of the Black-Scholes model (especially with the same trajectory for $W^1$) with constant volatility $\theta$, we approximate the price of the option by
$$
\bar{\nu}^{(N)}(\xi(\omega),F)-\bar{\nu}^{(N)}(\xi^{BS}(\omega),F)+C_{{\rm bar}}^{BS}(r,\sqrt{\theta},T,K,L)$$
where $C_{{\rm bar}}^{BS}$ denotes the (explicit) price of the up-and-out barrier option in the Black-Scholes model. Doing so with a simulation size $N=2.10^8$, we get the following accurate approximation of the  premium:
$$
e^{-rT}\ES[\left(S_T-K\right)_+ {\bf 1}_{\{\sup_{0\le t\le T} S_t\le L\}}]\approx 1,689.
$$
Then, setting $N=5.10^5$, we proceed $M=10^4$ independent Monte Carlo simulations of ${\cal E}_N$.  We denote  by $\bar{\sigma}_{_F}^2$ the empirical variance of the sample $({\cal E}_N^1,\ldots,{\cal E}_N^{M})$ (which corresponds to an estimation of $\sigma_{_F}^2$).
In Figure~\ref{f5} are depicted the density of a centered Gaussian random variable with variance $\bar{\sigma}_{_F}^2$ and the empirical density $\hat{f}_h$
(smoothed by a convolution with a Gaussian kernel) defined by:
$$
 \hat{f}_h(x)=\frac{1}{Mh}\sum_{\ell=1}^M\frac{1}{\sqrt{2\pi}}\exp\left(-\frac{(x-{\cal E}_N^{(\ell)})^2}{2h^2}\right).
 $$

\begin{figure}
\begin{center}
\includegraphics[width=7cm]{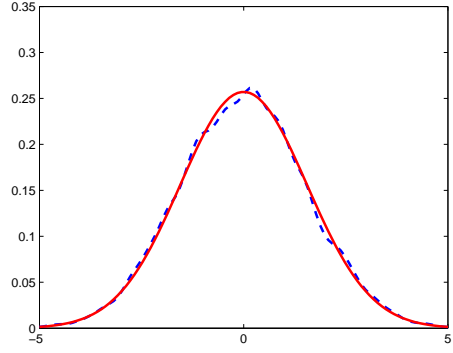}
\caption{Comparaison of the approximate density $\hat{f}_h$ of ${\cal E}_N$ (dotted line) with the density of ${\cal N}(0,\bar{\sigma}_F^2)$,
 $N=5.10^5$, $M=5.10^3$, $h=M^{-\frac{1}{5}}$.\label{f5}}
\end{center}\end{figure}

\medskip
As a conclusion, this numerical experiment first  illustrates that the $CLT$ occurs at a reasonable range (for numerical purpose) and also suggests that a local version holds true as well (``convergence of the density"). Another extension of our result could be, in the spirit of Bhattacharia's  result in~\cite{bhatta82} to establish an invariance principle of Donsker type.
\section*{Appendix}
\subsection*{A. Proof of identity~\eqref{eq:sigmacarre2}.} 
We have to deduce~\eqref{eq:sigmacarre2} from~\eqref{eq:sigmacarre}. First, 
we have (dropping $^x$ in $A^x_t$):
\begin{align*}
\ES_\nu\left[\left(\ES[A_{_{2T}}|{\cal F}_{_{2T}}]-\int_T^{2T}\sigma^*\nabla g_{_F}(X_u^x) dW_u\right)\ES[A_{_T}|{\cal F}_{_T}]\right]&=\ES_\nu\left[\ES[A_{_{2T}}|{\cal F}_{_{2T}}]\ES[A_{_T}|{\cal F}_{_T}]\right]\\
&=\ES_\nu[(\ES[A_{_T}|{\cal F}_{_T}])^2]
\end{align*}
since one easily checks that $\ES[A_{_{2T}}-A_{_T}|{\cal F}_{_T}]=0$. It follows that
\begin{align*}
T\sigma^2_{_F}=\ES_\nu\left[\left(\ES[A_{_{2T}}|{\cal F}_{_{2T}}]-\int_T^{2T}\sigma^*\nabla g_F(X_u)dW_u\right)^2\right]-\ES_\nu[(\ES[A_{_T}|{\cal F}_{_T}])^2].
\end{align*}
Second, using the Markov property (or the fact that $X^{(u),x}=\varphi(X_u^x,W^{(u)})$) and the stationarity of the process, one observes that
$\ES[A_{_{2T}}-A_{_T}|{\cal F}_{_{2T}}]$ and $\ES[A_{_T}|{\cal F}_{_T}]$ have the same distribution under $\PE_\nu$. In particular,
$$
\ES_\nu\left[\left(\ES[A_{_{2T}}-A_{_T}|{\cal F}_{_{2T}}]\right)^2\right]=\ES_\nu\left[\left(\ES[A_{_T}|{\cal F}_{_T}]\right)^2\right].
$$
Since $\ES[A_{_{2T}}|{\cal F}_{_{2T}}]=A_{_T}+\ES[A_{_{2T}}-A_{_T}|{\cal F}_{_{2T}}]$ and $\ES_\nu[A_{_T}\ES[A_{_{2T}}-A_{_T}|{\cal F}_{_{2T}}]]=0$,
we obtain that
$$
T\sigma^2_{_F}=\ES_\nu[A_{_T}^2]-2\ES_\nu\left[\ES[A_{_{2T}}|{\cal F}_{_{2T}}]\int_T^{2T}\sigma^*\nabla g_{_F}(X_u)dW_u\right]+\ES_\nu\left[\left(\int_T^{2T}\sigma^*\nabla g_{_F}(X_u)dW_u\right)^2\right].
$$
All we have to do now is to check that the three above terms correspond respectively to the three parts of~\eqref{eq:sigmacarre2}. First, by Fubini's Theorem,
$$
\ES_\nu[A_{_T}^2]=\int_{u=0}^T\int_{v=0}^T\ES_\nu[(F(X^{(u)}-f_{_F}(X_u))(F(X^{(v)}-f_{_F}(X_v)))]dv.
$$
Owing to the stationarity of the process under $\PE_\nu$, we have
$$
\ES_\nu[(F(X^{(u)}-f_{_F}(X_u))(F(X^{(v)}-f_{_F}(X_v)))]=C_{_F}(|u-v|),
$$ where $C_{_F}$ is defined by~\eqref{covariancefunction}.
This yields
$$
\ES_\nu[A_{_T}^2]=2\int_0^T\int_0^u C_{_F}(u-v)dv du=2\int_0^T (T-u)C_{_F}(u)du.
$$
Second, setting 
\begin{equation}\label{eq:mtf}
M_T^f=\int_0^t\sigma^*\nabla f(X_u)dW_u,
\end{equation}
we have
\begin{align*}
\ES_\nu\left[\ES[A_{_{2T}}|{\cal F}_{_{2T}}]\int_T^{2T}\sigma^*\nabla g_{_F}(X_u)dW_u\right]=&\int_0^{2T}\ES_\nu[(F_{_T}(X^{(u)})-f_{_F}(X_u))(M_{_{2T}}^{g_{_F}}-M_{_T}^{g_{_F}})] du\\
=&\int_0^T\ES_\nu[(F_{_T}(X^{(u)})-f_{_F}(X_u))(M_{_{T+u}}^{g_{_F}}-M_{_T}^{g_{_F}})]du\\
&+\int_T^{2T}\ES_\nu[(F_{_T}(X^{(u)})-f_{_F}(X_u))(M_{_{2T}}^{g_{_F}}-M_{_u}^{g_{_F}})]du.
\end{align*}
Now, the fact that $M_{_{T+u}}^{g_{_F}}-M_{_T}^{g_{_F}}=g_{_F}(X_{T+u})-g_{_F}(X_{T})-\int_T^{T+u}{\cal A} g_F(X_v)dv$ implies that we can make use of the stationarity property to obtain for every $u\in[0,T]$,
\begin{align*}
\ES_\nu[(F_{_T}(X^{(u)})-f_{_F}(X_u))&(M_{_{T+u}}^{g_{_F}}-M_{_T}^{g_{_F}})]\\
&=\ES_\nu\left[\left(F_{_T}(X)-f_{_F}(X_0)\right)\left(g_{_F}(X_{T})-g_{_F}(X_{0})-\int_0^{T}{\cal A} g_F(X_v)dv\right)\right]\\
&=\ES_\nu[(F_{_T}(X)-f_{_F}(X_0))(M_T^{g_{_F}}-M_{T-u}^{g_{_F}})].
\end{align*}
With similar arguments, one checks that for every $u\in[T,2T]$,
$$
\ES_\nu[(F_{_T}(X^{(u)})-f_{_F}(X_u))(M_{_{2T}}^{g_{_F}}-M_{_u}^{g_{_F}})]=\ES_\nu[(F_{_T}(X)-f_{_F}(X_0))M_{_{2T-u}}^{g_{_F}}].
$$
It follows that
\begin{align*}
&\ES_\nu\left[\ES[A_{_{2T}}|{\cal F}_{_{2T}}]\int_T^{2T}\sigma^*\nabla g_{_F}(X_u)dW_u\right]\\&=
\ES_\nu\left[\left(F_{_T}(X)-f_{_F}(X_0)\right)\left(TM_T^{g_{_F}}-\int_0^T M_{T-u}^{g_{_F}}du+\int_T^{2T}M_{_{2T-u}}^{g_{_F}}du\right)\right]=T\ES_\nu[F_{_T}(X)M_T^{g_{_F}}].
\end{align*}
\begin{align*}
\mbox{Finally,  }\qquad \ES_\nu\left[\left(\int_T^{2T}\sigma^*\nabla g_{_F}(X_u)dW_u\right)^2\right]&=\int_T^{2T}\ES_\nu\left[\left|\sigma^*\nabla g_F(X_u)\right|^2\right]du\qquad\qquad\qquad\\
&=T\int|\sigma^*\nabla g_F(x)|^2\nu(dx)
\end{align*}
owing to the stationarity of the process. This concludes the proof.\hfill $\Box$

\subsection*{B.  Computation of $\sigma^2_{_F}$ when $F(\alpha)=\phi(\alpha_{_T})$.} As mentioned in~\eqref{remarquecalcul}, when $\phi={\cal A}h+C$, the $CLT$ for marginal functions combined with a change of variable yields $\sigma^2_{_F}=\int_{\ER^d} |  \sigma^*\nabla h (x) |^2\nu(dx)$. Let us check this formula starting from~\eqref{eq:sigmacarre}. Following the notation introduced in~\eqref{eq:mtf}, we have
\begin{align*}
&\ES[A_{_{2T}}|{\cal F}_{_{2T}}]-\ES[A_{_T}|{\cal F}_{_T}]-\int_T^{2T}\sigma^*\nabla g_{_F}(X_u^x)dW_u= \varphi_1(x,.)-\varphi_2(x,.)\quad\textnormal{where,}\\
&\varphi_1(x,.)=\int_0^T \phi(X_{u+T}^x)du +\int_T^{2T}\ES\left[\phi(X_{u+T}^x) {\cal F}_{2T}\right]du-\int_T^{2T}f_{_F}(X_u^x)du-(M_{2T}^{g_{_F}}-M_T^{g_{_F}})\quad\textnormal{and}\\
&\varphi_2(x,.)=\int_0^T\ES\left[\phi(X_{u+T}^x)\,|\, {\cal F}_{T}\right]du.
\end{align*}
In this case, $f_{_F}=P_{_T}\phi$ and using that ${\cal A}$ and $P_{_T}$ commute, one checks that  $f_{_F}-\nu(f_{_F})= {\cal A} P_{_T} h$. This implies that $g_{_F}=P_{_T}\phi$. For the sake of simplicity we may assume w.l.g.  $\nu(f_{_F})=\nu(\phi)=0$. Then, on the one hand 
\begin{eqnarray*}
\varphi_1(x,.) &=& \int_T^{2T} {\cal A}h(X_u)du+ \int_T^{2T}P_{u-T}\phi(X_{2T})du-\Big[g_{_F}(X_{2T})-g_{_F}(X_{_T})\Big]\\
&=& h(X_{2T})-h(X_T)-(M_{2T}^h-M_T^h)+\int_0^T{\cal A}P_u h(X_{2T})du -\Big[g_{_F}(X_{2T})-g_{_F}(X_{_T})\Big]\\
&=& h(X_{2T})-h(X_T)-(M_{2T}^h-M_T^h)+\underbrace{P_{_T}h(X_{2T})}_{=g_{_F}(X_{2T})}-h(X_{2T}) -\Big[g_{_F}(X_{2T})-g_{_F}(X_{_T})\Big]\\
&=& g_{_F}(X_T)-h(X_T)-(M_{2T}^h-M_T^h).
\end{eqnarray*}
On the other hand
\begin{equation*}
\varphi_2(x,.)= \int_0^TP_u\phi(X_T)du= \int_0^T {\cal A}P_u h(X_T)du= P_{_T}h(X_T)-h(X_T)=g_{_F}(X_T)-h(X_T).\\
\end{equation*}
so that
$$
\sigma^2_{_F}=\frac{1}{T}\ES_\nu[(M_{2T}^h-M_T^h)^2]=\frac{1}{T}\int_T^{2T}\ES_\nu[|\sigma^*\nabla h(X_u)|^2]du=\int|\sigma^*\nabla h(x)|^2\nu(dx).\qquad \Box
$$
\small

\end{document}